\documentclass[11pt,a4paper]{article}

\usepackage[T1]{fontenc}
\usepackage[english]{babel}
\usepackage{amsmath,amssymb,amsthm,mathtools}
\usepackage{geometry}
\usepackage{booktabs}
\usepackage{array}
\usepackage{enumitem}
\usepackage{hyperref}
\hypersetup{colorlinks=true,linkcolor=blue,citecolor=blue,urlcolor=blue}
\usepackage{cleveref}

\usepackage{tikz}
\usetikzlibrary{calc,arrows.meta}
\usepackage{float}
\usepackage{pgfplots}
\pgfplotsset{compat=1.17}

\theoremstyle{plain}
\newtheorem{theorem}{Theorem}[section]
\newtheorem{lemma}[theorem]{Lemma}
\newtheorem{proposition}[theorem]{Proposition}
\newtheorem{corollary}[theorem]{Corollary}

\theoremstyle{definition}
\newtheorem{definition}[theorem]{Definition}

\newtheorem{remark}[theorem]{Remark}

\crefname{hypothesis}{Hypothesis}{Hypotheses}
\Crefname{hypothesis}{Hypothesis}{Hypotheses}
\crefname{openproblem}{Open Problem}{Open Problems}
\Crefname{openproblem}{Open Problem}{Open Problems}
\crefname{assumption}{Assumption}{Assumptions}
\Crefname{assumption}{Assumption}{Assumptions}

\newcommand{\OO}{\mathcal{O}}

\title{Elliptic Superlevel Sets in the Class $\OO_C$:\\[1mm]
Boundary-Layer Geometry and Stability of the Return Map}
\author{M. Barkatou\thanks{ISTM Laboratory, Chouaib Doukkali University, Morocco. \texttt{barkatou.m@ucd.ac.ma}}
\and
M. El Morsalani\thanks{QWAVE-CONSULT.EU. \texttt{Mohamed.elmorsalani@qwave-consult.eu}}}
\date{}

\begin{document}
\maketitle

\begin{abstract}
We study the interaction between elliptic superlevel geometry and the return dynamics associated with the class \(\mathcal O_C\). Let \(C\subset\mathbb R^N\) be a compact strictly convex core and let \(\Omega\in\mathcal O_C\) be a bounded domain whose boundary is a global normal graph over \(\partial C\) and whose inward normals return to the interior of \(C\). For the Dirichlet problem
\[
-\Delta u=f \quad \text{in }\Omega,
\qquad
u=0 \quad \text{on }\partial\Omega,
\]
with \(f\ge0\), \(f\not\equiv0\), and \(\operatorname{supp}f\Subset\operatorname{int}C\), we consider the superlevel domains
\[
\Omega_t=\{u>t\},
\qquad t\downarrow0.
\]

We prove that, for all sufficiently small \(t>0\), the boundary \(\partial\Omega_t\) is a global normal graph
\[
\partial\Omega_t
=
\{c+d_t(c)\nu(c):c\in\partial C\},
\]
with
\[
\|d_t-d_0\|_{C^1(\partial C)}=O(t),
\qquad
d_t=d_0-tb_0+o_{C^1}(t),
\]
where \(b_0\) is determined by the Hopf boundary derivative and the transversality between the outer boundary and the normal rays of \(C\). Using exact normal-graph formulas, we derive the corresponding first variation of the moving unit normal and show that the geometric normal property persists:
\[
\Omega_t\in\mathcal O_C
\]
for all sufficiently small \(t\).

The associated return maps \(F_t:\partial C\to\partial C\) therefore remain well defined. We establish the quantitative stability estimate
\[
\|F_t-F_0\|_{L^\infty(\partial C)}=O(t)
\]
and the first-order expansion
\[
F_t=F_0+t\mathcal A_0+o_{L^\infty}(t),
\]
where \(\mathcal A_0\) is obtained from the PDE-generated boundary displacement and normal variation by the return transversality constraint. Finally, for any sequence \(t_n\downarrow0\), the non-autonomous system
\[
x_{n+1}=F_{t_n}(x_n)
\]
is asymptotically autonomous with background map \(F_0\), and summable level sequences generate summable one-step perturbation defects. The results provide a quantitative bridge between elliptic boundary-layer deformation and return-map dynamics in normal-graph geometry relative to a convex core.
\end{abstract}

non-autonomous dynamical systems; geometric normal property; boundary layers.


\tableofcontents
\section{Introduction}
\label{sec:introduction}

Geometric information carried by the level sets of elliptic equations has
long played an important role in the interaction between partial
differential equations and convex geometry.  Classical results concern,
among other questions, convexity and quasiconcavity of level sets,
capacitary functions in convex rings, and geometric properties of
solutions to elliptic boundary-value problems
\cite{Lewis1977,Korevaar1990,MaOu2010}.  These questions remain active:
recent work continues to investigate convexity of solutions and level
sets for semilinear, free-boundary, and fully nonlinear elliptic
problems
\cite{MebratePorru2024,Bensid2025,ChenJiaXiong2025}.
Related geometric perturbation questions also form part of the broader
framework of shape analysis and boundary deformation
\cite{DelfourZolesio2001}.

The problem considered here is of a different nature.  Instead of asking
whether the level sets themselves satisfy a prescribed convexity
property, we ask whether an elliptically generated family of level
domains preserves a geometric class carrying its own return dynamics,
and how the associated return maps vary with the level.

Let
\[
C\subset\mathbb R^N,
\qquad N\geq2,
\]
be a fixed compact strictly convex body, and let
\[
\Omega\in\mathcal O_C
\]
be an admissible domain containing \(C\).  The class
\(\mathcal O_C\), introduced in the geometric setting of
\cite{Barkatou2002}, is characterized here by two complementary
properties.  First, the outer boundary is a global normal graph over
\(\partial C\):
\[
\partial\Omega
=
\{c+d_0(c)\nu(c):c\in\partial C\},
\]
where \(\nu\) denotes the outward unit normal of the core and
\(d_0>0\) is the background thickness.  Second, every inward normal ray
issued from \(\partial\Omega\) reaches the interior of \(C\).  The latter
geometric normal property makes it possible to associate with the pair
\((C,\Omega)\) a return map
\[
F_0:\partial C\longrightarrow\partial C:
\]
one first follows the outward normal of the core from
\(c\in\partial C\) to \(\partial\Omega\), and then follows the inward
normal of \(\partial\Omega\) back to its first intersection with
\(\partial C\).

This return construction and its fixed-domain dynamics were developed in
our preceding work
\cite{BarkatouElMorsalaniReturn2026,ElMorsalaniBarkatou2026a}.
The present paper addresses a different question.  The outer boundary is
no longer regarded as fixed.  It is deformed by the level sets of an
elliptic solution.

More precisely, we consider
\begin{equation}
\begin{cases}
-\Delta u=f & \text{in }\Omega,\\
u=0 & \text{on }\partial\Omega,
\end{cases}
\label{eq:intro-dirichlet}
\end{equation}
where
\[
f\in C^{0,\alpha}(\overline\Omega),
\qquad
f\geq0,
\qquad
f\not\equiv0,
\qquad
\operatorname{supp}f\Subset\operatorname{int}C.
\]
For \(t>0\), let
\begin{equation}
\Omega_t:=\{x\in\Omega:u(x)>t\}.
\label{eq:intro-superlevel}
\end{equation}
Since \(u=0\) on \(\partial\Omega\), the small-level regime
\[
t\downarrow0
\]
moves \(\partial\Omega_t\) toward the fixed outer boundary
\(\partial\Omega\), rather than toward the core.  The natural geometric
question is therefore:

\begin{quote}
	For sufficiently small \(t>0\), does \(\Omega_t\) remain in
	\(\mathcal O_C\), and, if so, how does its return map \(F_t\) compare
	with the background map \(F_0\)?
\end{quote}

The answer requires more than regularity of the level hypersurfaces.
Membership in \(\mathcal O_C\) is a global geometric condition: one must
control both the normal-graph representation of \(\partial\Omega_t\)
and the inward-normal return to the core.  Moreover, quantitative
stability of \(F_t\) depends simultaneously on the displacement of the
level boundary, the variation of its normal field, and the
transversality of the return intersection with \(\partial C\).

The basic mechanism begins at the outer boundary.  The Hopf lemma,
combined with the transversality between the fixed outer boundary and
the normal rays of \(C\), yields a uniform collar in which
\[
\partial_r u(c,r)<0.
\]
Consequently, every sufficiently small level boundary possesses a unique
global representation
\[
\partial\Omega_t
=
\{c+d_t(c)\nu(c):c\in\partial C\}.
\]
The small geometric quantity is not \(d_t\) itself but
\[
h_t:=d_0-d_t.
\]
We prove
\[
h_t=O(t)
\]
and identify its first-order profile:
\[
d_t
=
d_0-tb_0+o_{C^1}(t),
\]
where
\[
b_0(c)
=
\frac{1}{
	-\partial_{\nu_\Omega}u(\Psi_0(c))\,
	\nu_\Omega(\Psi_0(c))\cdot\nu(c)
}.
\]
Thus the first-order motion of the level surface contains both an
elliptic contribution, through the Hopf boundary derivative, and a
geometric contribution, through the angle between the two normal
fields.

A point that is important throughout the paper is the distinction
between exact geometric identities and asymptotic expansions.  With the
shape-operator convention
\[
S=-D\nu|_{T\partial C},
\]
we first derive the exact normal-graph formulas.  In particular, if
\[
L_t=I-d_tS,
\qquad
w_t=L_t^{-1}\nabla_{\partial C}d_t,
\]
then the outward unit normal of \(\partial\Omega_t\) is exactly
\[
N_t
=
\frac{\nu-w_t}{\sqrt{1+|w_t|^2}}.
\]
Only after this identity has been established do we expand the moving
geometry in the parameter \(t\).  This order of argument avoids
introducing independent ``normal correction'' terms and keeps the
sign conventions explicit.

The principal contributions of the paper are the following.

\begin{enumerate}
	\item
	\textbf{Elliptic boundary-layer representation.}
	We prove that all sufficiently small superlevel boundaries are contained
	in a fixed outer collar and are unique global normal graphs over
	\(\partial C\).  Their thickness functions satisfy
	\[
	\|d_t-d_0\|_{C^1(\partial C)}=O(t)
	\]
	and
	\[
	d_t=d_0-tb_0+o_{C^1}(t).
	\]
	
	\item
	\textbf{Exact moving-boundary geometry.}
	We derive exact formulas for the differential, induced metric, surface
	Jacobian, and inward and outward unit normals of
	\(\partial\Omega_t\).  From the first variation of \(d_t\) we then
	obtain the corresponding first variation of the normal field.
	
	\item
	\textbf{Persistence of the class \(\mathcal O_C\).}
	Using the uniform transversality of the background return and a
	quantitative penetration argument into the core, we prove that there
	exists \(t_{\mathcal O}>0\) such that
	\[
	\Omega_t\in\mathcal O_C,
	\qquad
	0<t<t_{\mathcal O}.
	\]
	Thus the return construction survives the elliptically generated
	boundary deformation.
	
	\item
	\textbf{Stability and first variation of the return map.}
	Let \(F_t\) denote the return map generated by
	\((C,\Omega_t)\).  We establish
	\[
	\|F_t-F_0\|_{L^\infty(\partial C)}=O(t)
	\]
	and derive the first-order expansion
	\[
	F_t
	=
	F_0+t\mathcal A_0+o_{L^\infty}(t).
	\]
	The coefficient \(\mathcal A_0\) is obtained by combining the
	PDE-generated displacement of the outer boundary with the variation of
	its inward normal and projecting the resulting displacement onto
	\(T_{F_0(c)}\partial C\) along the return direction.
	
	\item
	\textbf{PDE-generated non-autonomous dynamics.}
	For any sequence \(t_n\downarrow0\), we consider
	\[
	x_{n+1}=F_{t_n}(x_n).
	\]
	The family \(F_{t_n}\) is asymptotically autonomous with background map
	\(F_0\):
	\[
	\sup_{c\in\partial C}
	d_C\bigl(F_{t_n}(c),F_0(c)\bigr)
	=O(t_n).
	\]
	Summable level sequences therefore generate summable one-step
	perturbation defects, and we obtain finite-horizon orbit-tracking
	estimates.  Longer-time conclusions are stated only under additional
	dynamical assumptions on the background return map.
	
	\item
	\textbf{Explicit benchmark and computational protocol.}
	We include an exactly solvable radial benchmark, for which the
	boundary-layer expansion can be verified directly, and a non-radial
	normal-graph example together with a computational protocol for
	\(d_t\), the moving normals, and the return maps.  The protocol is
	stated explicitly so that the predicted convergence rates can be tested
	independently.
\end{enumerate}

The role of the elliptic equation in these results is therefore
specific.  It generates a controlled deformation of the outer boundary;
the subsequent normal and return-map variations are geometric
consequences of that deformation.  In schematic form, the mechanism is
\[
\boxed{
	-\Delta u=f,\quad u|_{\partial\Omega}=0
	\;\Longrightarrow\;
	d_t\to d_0
	\;\Longrightarrow\;
	\Omega_t\in\mathcal O_C
	\;\Longrightarrow\;
	F_t\to F_0.
}
\]
The first-order results obtained below quantify each arrow in this
chain.

The paper is organized as follows.  Section~\ref{sec:geometry}
introduces the fixed geometry of the class \(\mathcal O_C\) and the
background return construction.  Section~\ref{sec:pde-collar}
establishes the uniform elliptic boundary collar.
Section~\ref{sec:superlevel-normal-graphs} studies the superlevel
boundaries and derives the first-order expansion of \(d_t\).
Section~\ref{sec:exact-moving-boundary-geometry} develops the exact geometry
and first variation of the moving boundary.
Section~\ref{sec:stability-OC} proves persistence of
\(\mathcal O_C\).
Section~\ref{sec:return-stability} establishes stability and first-order
asymptotics of the return map.
Section~\ref{sec:nonautonomous-return-system} considers the induced
non-autonomous system.  Section~\ref{sec:numerical-illustrations}
provides the radial benchmark and non-radial computational protocol, and
Section~\ref{sec:open-problems} records several extensions and open
questions.

\section{Geometric framework and the class $\OO_C$}
\label{sec:geometry}
\subsection{Reference core and normal coordinates}
\label{subsec:core-normal-coordinates}

Let
\[
C\subset\mathbb{R}^{N},
\qquad N\geq2,
\]
be a compact strictly convex body with nonempty interior and boundary
\[
\partial C\in C^{2,\alpha},
\qquad 0<\alpha<1.
\]
We denote by
\[
\nu:\partial C\longrightarrow\mathbb S^{N-1}
\]
the outward unit normal field.

Throughout the paper we use the shape-operator convention
\begin{equation}
S_c:=-D\nu(c)\big|_{T_c\partial C},
\qquad c\in\partial C.
\label{eq:shape-operator-convention}
\end{equation}
Thus, if
\[
\kappa_1(c),\ldots,\kappa_{N-1}(c)\geq0
\]
denote the usual principal curvatures of the convex hypersurface
\(\partial C\), then the eigenvalues of \(S_c\) are
\[
-\kappa_1(c),\ldots,-\kappa_{N-1}(c).
\]
In particular,
\[
S_c\leq0
\]
as a symmetric endomorphism of \(T_c\partial C\).

For \(r\geq0\), consider the outward normal-coordinate map
\begin{equation}
\Phi(c,r):=c+r\nu(c).
\label{eq:normal-coordinate-map}
\end{equation}
For every \(\tau\in T_c\partial C\) and \(s\in\mathbb R\),
\begin{equation}
D\Phi(c,r)(\tau,s)
=
(I-rS_c)\tau+s\nu(c).
\label{eq:differential-normal-coordinate-map}
\end{equation}
Indeed, by \eqref{eq:shape-operator-convention},
\[
D\nu(c)\tau=-S_c\tau.
\]
Consequently,
\[
I-rS_c
\]
has eigenvalues
\[
1+r\kappa_i(c)\geq1>0,
\qquad i=1,\ldots,N-1,
\]
and is therefore invertible for every \(r\geq0\).

Since \(C\) is convex, every point \(x\in\mathbb R^N\setminus C\)
has a unique metric projection \(P_C(x)\in C\).  Because
\(\partial C\) is \(C^1\), if \(x\notin C\) then
\[
P_C(x)\in\partial C
\]
and
\[
x-P_C(x)
=
\operatorname{dist}(x,C)\,
\nu(P_C(x)).
\]
Hence every exterior point possesses a unique representation
\begin{equation}
x=\Phi(c,r),
\qquad
c\in\partial C,\quad r>0.
\label{eq:unique-exterior-normal-coordinate}
\end{equation}

For later use, if
\[
d:\partial C\longrightarrow(0,\infty)
\]
is a sufficiently regular function, we write
\begin{equation}
\Gamma_d
:=
\bigl\{
c+d(c)\nu(c):c\in\partial C
\bigr\}
\label{eq:normal-graph}
\end{equation}
for its normal graph over \(\partial C\), and
\begin{equation}
\Psi_d(c)
:=
c+d(c)\nu(c)
\label{eq:normal-graph-parametrization}
\end{equation}
for the corresponding parametrization.

Its differential is
\begin{equation}
D\Psi_d(c)\tau
=
(I-d(c)S_c)\tau
+
\bigl(\nabla_{\partial C}d(c)\cdot\tau\bigr)\nu(c),
\qquad
\tau\in T_c\partial C.
\label{eq:differential-normal-graph}
\end{equation}

We shall use the notation
\begin{equation}
L_d(c):=I-d(c)S_c.
\label{eq:Ld-definition}
\end{equation}
By convexity, \(S_c\leq0\), and therefore
\[
L_d(c)=I-d(c)S_c\geq I>0
\]
for every nonnegative thickness function \(d\). In particular,
\(L_d(c)\) is invertible without any small-thickness assumption.
\subsection{The class $\mathcal O_C$ and the background thickness}
\label{subsec:class-OC}

We now fix the geometric class in which the return dynamics is defined.
Throughout this subsection, the reference core \(C\) is the strictly convex
body introduced in Subsection~\ref{subsec:core-normal-coordinates}.

\begin{definition}[The class $\mathcal O_C$]
	\label{def:class-OC}
	Let
	\[
	\Omega\subset\mathbb R^N
	\]
	be a bounded open set with
	\[
	C\subset\Omega,
	\qquad
	\partial\Omega\in C^{2,\alpha}.
	\]
	We say that
	\[
	\Omega\in\mathcal O_C
	\]
	if the following two geometric properties hold.
	
	\smallskip
	\noindent
	\textnormal{(OC1) Outward normal graph property.}
	For every \(c\in\partial C\), there exists a unique number
	\[
	d_0(c)>0
	\]
	such that
	\begin{equation}
	\bigl\{
	r\geq0:
	c+r\nu(c)\in\Omega
	\bigr\}
	=
	[0,d_0(c)).
	\label{eq:OC-outward-ray}
	\end{equation}
	In addition,
	\[
	\bigl\{
	r\geq0:
	c+r\nu(c)\in\overline{\Omega}
	\bigr\}
	=
	[0,d_0(c)].
	\]

	\smallskip
	\noindent
	\textnormal{(OC2) Geometric normal property.}
	For every \(x\in\partial\Omega\), the inward normal ray
	\begin{equation}
	\bigl\{
	x-s\,\nu_\Omega(x):
	s\geq0
	\bigr\}
	\label{eq:OC-inward-ray}
	\end{equation}
	intersects the interior of \(C\), where
	\(\nu_\Omega(x)\) denotes the outward unit normal to
	\(\partial\Omega\) at \(x\).
\end{definition}

The function
\begin{equation}
d_0:\partial C\longrightarrow(0,\infty)
\label{eq:background-thickness}
\end{equation}
defined by \eqref{eq:OC-outward-ray} will be called the
\emph{background thickness function} of \(\Omega\) relative to \(C\).

Condition \textnormal{(OC1)} has an immediate global consequence.
By the uniqueness of the exterior normal coordinates established in
Subsection~\ref{subsec:core-normal-coordinates},
\begin{equation}
\Omega\setminus C
=
\bigl\{
\Phi(c,r):
c\in\partial C,\;
0<r<d_0(c)
\bigr\},
\label{eq:OC-shell-representation}
\end{equation}
and
\begin{equation}
\partial\Omega
=
\Gamma_{d_0}
=
\bigl\{
c+d_0(c)\nu(c):
c\in\partial C
\bigr\}.
\label{eq:background-boundary-normal-graph}
\end{equation}
Accordingly, we write
\begin{equation}
\Psi_0(c)
:=
c+d_0(c)\nu(c),
\qquad c\in\partial C,
\label{eq:background-parametrization}
\end{equation}
for the background boundary parametrization.

Since \(C\Subset\Omega\) and both sets are bounded, the thickness is
bounded away from zero and from infinity:
\begin{equation}
0<d_- \leq d_0(c)\leq d_+<\infty,
\qquad c\in\partial C,
\label{eq:background-thickness-bounds}
\end{equation}
for suitable constants \(d_-,d_+>0\).

The two conditions in Definition~\ref{def:class-OC} play different roles.
Condition \textnormal{(OC1)} describes the domain globally as a single-valued
normal graph over the reference core.  Condition \textnormal{(OC2)} is an
additional dynamical requirement: it guarantees that the inward normal
issued from the outer boundary reaches the core.  In particular,
\textnormal{(OC2)} is not merely a reformulation of the graph property.

For later comparison with the moving PDE level sets, it is convenient to
separate the domain from its thickness function.  If
\(d:\partial C\to(0,\infty)\) is a positive function, we set
\begin{equation}
\Omega[d]
:=
C
\cup
\bigl\{
\Phi(c,r):
c\in\partial C,\;
0<r<d(c)
\bigr\}.
\label{eq:domain-generated-by-thickness}
\end{equation}
Thus the fixed background domain satisfies
\begin{equation}
\Omega=\Omega[d_0].
\label{eq:background-domain-thickness}
\end{equation}
The central geometric question for the PDE-generated family will be whether
its thickness functions \(d_t\) remain in the same class, namely whether
\[
\Omega[d_t]\in\mathcal O_C
\]
for all sufficiently small \(t>0\).

\subsection{Regularity of the background thickness and exact normal formula}
\label{subsec:background-regularity-normal}

We first show that the geometric normal property provides precisely the
transversality needed to recover the regularity of the background thickness
function introduced in Definition~\ref{def:class-OC}.

\begin{lemma}[Transversality of the background normal graph]
	\label{lem:background-transversality}
	Let $\Omega\in\mathcal O_C$, and let
	\[
	\Psi_0(c)=c+d_0(c)\nu(c),
	\qquad c\in\partial C,
	\]
	be the parametrization of $\partial\Omega$ given by
	\eqref{eq:background-parametrization}. Then $d_0$ is continuous on
	$\partial C$, and
	\begin{equation}
	\nu_\Omega(\Psi_0(c))\cdot\nu(c)>0
	\qquad
	\text{for every }c\in\partial C.
	\label{eq:background-pointwise-transversality}
	\end{equation}
	Consequently, there exists a constant $\theta_0>0$ such that
	\begin{equation}
	\nu_\Omega(\Psi_0(c))\cdot\nu(c)
	\geq\theta_0
	\qquad
	\text{for all }c\in\partial C.
	\label{eq:background-uniform-transversality}
	\end{equation}
\end{lemma}

\begin{proof}
	We begin with continuity of $d_0$. Let $c_j\to c$ in $\partial C$.
	Since $\Omega$ is bounded, the sequence $d_0(c_j)$ is bounded. Let
	$d_0(c_{j_\ell})\to r$ along a convergent subsequence. Since $\nu$ is
	continuous,
	\[
	\Psi_0(c_{j_\ell})
	=
	c_{j_\ell}
	+
	d_0(c_{j_\ell})\nu(c_{j_\ell})
	\longrightarrow
	c+r\nu(c).
	\]
	Each point $\Psi_0(c_{j_\ell})$ belongs to $\partial\Omega$, and
	$\partial\Omega$ is closed. Hence
	\[
	c+r\nu(c)\in\partial\Omega.
	\]
	By the uniqueness in condition \textnormal{(OC1)}, the only boundary point
	of $\Omega$ lying on the outward normal ray issued from $c$ is
	\[
	c+d_0(c)\nu(c).
	\]
	Therefore $r=d_0(c)$. Since every convergent subsequence has the same limit,
	\[
	d_0(c_j)\longrightarrow d_0(c),
	\]
	and $d_0$ is continuous.
	
	Fix now $c\in\partial C$ and put
	\[
	x=\Psi_0(c)
	=
	c+d_0(c)\nu(c).
	\]
	By condition \textnormal{(OC2)}, there exists $s>0$ such that
	\[
	y:=x-s\nu_\Omega(x)\in\operatorname{int}C.
	\]
	Because $\nu(c)$ is the outward supporting normal of the convex body $C$
	at $c$, every $y\in\operatorname{int}C$ satisfies
	\[
	(y-c)\cdot\nu(c)<0.
	\]
	Using
	\[
	y-c
	=
	d_0(c)\nu(c)-s\nu_\Omega(x),
	\]
	we obtain
	\[
	d_0(c)
	-
	s\,\nu_\Omega(x)\cdot\nu(c)
	<0.
	\]
	Thus
	\[
	\nu_\Omega(x)\cdot\nu(c)
	>
	\frac{d_0(c)}{s}
	>0,
	\]
	which proves \eqref{eq:background-pointwise-transversality}.
	
	The map
	\[
	c\longmapsto
	\nu_\Omega(\Psi_0(c))\cdot\nu(c)
	\]
	is continuous on the compact manifold $\partial C$. Its pointwise
	positivity therefore implies the uniform bound
	\eqref{eq:background-uniform-transversality}.
\end{proof}

\begin{proposition}[Regularity of the background thickness]
	\label{prop:background-thickness-regularity}
	If $\Omega\in\mathcal O_C$, then
	\[
	d_0\in C^{1,\alpha}(\partial C).
	\]
	Moreover,
	\[
	\Psi_0:\partial C\longrightarrow\partial\Omega
	\]
	is a $C^{1,\alpha}$ diffeomorphism.
\end{proposition}

\begin{proof}
	Let $\rho_\Omega$ denote the signed distance to $\partial\Omega$, chosen
	positive outside $\Omega$ and negative inside. Since
	\[
	\partial\Omega\in C^{2,\alpha},
	\]
	the function $\rho_\Omega$ is of class $C^{2,\alpha}$ in a neighborhood
	of $\partial\Omega$, and
	\[
	\nabla\rho_\Omega(x)=\nu_\Omega(x),
	\qquad x\in\partial\Omega.
	\]
	
	Consider locally on $\partial C$ the function
	\[
	H(c,r)
	:=
	\rho_\Omega\bigl(c+r\nu(c)\bigr).
	\]
	The regularity $\partial C\in C^{2,\alpha}$ implies that
	\[
	(c,r)\longmapsto c+r\nu(c)
	\]
	is of class $C^{1,\alpha}$, hence so is $H$. By construction,
	\[
	H(c,d_0(c))=0.
	\]
	Furthermore,
	\[
	\partial_r H(c,d_0(c))
	=
	\nu_\Omega(\Psi_0(c))\cdot\nu(c)
	\geq\theta_0>0
	\]
	by Lemma~\ref{lem:background-transversality}.
	
	The implicit function theorem therefore yields
	\[
	d_0\in C^{1,\alpha}(\partial C).
	\]
	Consequently $\Psi_0$ is of class $C^{1,\alpha}$.
	
The representation
\[
\partial\Omega
=
\{\Psi_0(c):c\in\partial C\}
\]
is bijective by condition \textnormal{(OC1)}. It remains to verify
nondegeneracy of the differential. For \(\tau\in T_c\partial C\),
\[
D\Psi_0(c)[\tau]
=
L_0(c)\tau
+
\bigl(\nabla_{\partial C}d_0(c)\cdot\tau\bigr)\nu(c),
\qquad
L_0(c)=I-d_0(c)S_c.
\]
If \(D\Psi_0(c)[\tau]=0\), taking the component tangent to
\(T_c\partial C\) gives
\[
L_0(c)\tau=0.
\]
Because \(S_c\leq0\) and \(d_0(c)>0\), one has
\[
L_0(c)=I-d_0(c)S_c\geq I>0,
\]
so \(L_0(c)\) is positive definite and hence invertible. Therefore \(\tau=0\), so
\(D\Psi_0(c)\) is injective at every point. Since the domain and target
have the same dimension, the inverse function theorem makes \(\Psi_0\)
a local \(C^{1,\alpha}\) diffeomorphism. Together with global
bijectivity and compactness of \(\partial C\), this proves that
\(\Psi_0:\partial C\to\partial\Omega\) is a global
\(C^{1,\alpha}\) diffeomorphism.
\end{proof}

We now record the exact unit normal of an arbitrary normal graph. This is a
purely geometric identity and does not involve the elliptic equation.

\begin{proposition}[Exact normal of a normal graph]
	\label{prop:exact-normal-graph}
	Let
	\[
	d\in C^1(\partial C),
	\qquad d>0,
	\]
	and let
	\[
	\Gamma_d
	=
	\{c+d(c)\nu(c):c\in\partial C\}.
	\]
	Set
	\[
	L_d(c)
	:=
	I-d(c)S_c.
	\]
	Then the outward unit normal to $\Gamma_d$ at
	\[
	x=\Psi_d(c)=c+d(c)\nu(c)
	\]
	is given exactly by
	\begin{equation}
	N_d(c)
	=
	\frac{
		\nu(c)-L_d(c)^{-1}\nabla_{\partial C}d(c)
	}{
		\sqrt{
			1+
			\left|
			L_d(c)^{-1}\nabla_{\partial C}d(c)
			\right|^2
		}
	}.
	\label{eq:exact-normal-graph}
	\end{equation}
	The inward unit normal is therefore
	\begin{equation}
	n_d(c)=-N_d(c).
	\label{eq:exact-inward-normal-graph}
	\end{equation}
\end{proposition}

\begin{proof}
	For $\tau\in T_c\partial C$, equation
	\eqref{eq:differential-normal-graph} gives
	\[
	D\Psi_d(c)\tau
	=
	L_d(c)\tau
	+
	\bigl(
	\nabla_{\partial C}d(c)\cdot\tau
	\bigr)\nu(c).
	\]
	Since $S_c$ is symmetric, so are $L_d(c)$ and $L_d(c)^{-1}$.
	Define
	\[
	W_d(c)
	:=
	\nu(c)
	-
	L_d(c)^{-1}\nabla_{\partial C}d(c).
	\]
	Then
	\begin{align*}
	W_d(c)\cdot D\Psi_d(c)\tau
	&=
	\nabla_{\partial C}d(c)\cdot\tau
	-
	\left\langle
	L_d(c)^{-1}\nabla_{\partial C}d(c),
	L_d(c)\tau
	\right\rangle
	\\
	&=0.
	\end{align*}
	Hence $W_d(c)$ is normal to $\Gamma_d$.
	
	Because
	\[
	L_d(c)^{-1}\nabla_{\partial C}d(c)
	\in T_c\partial C
	\]
	and $\nu(c)\perp T_c\partial C$,
	\[
	|W_d(c)|^2
	=
	1+
	\left|
	L_d(c)^{-1}\nabla_{\partial C}d(c)
	\right|^2.
	\]
	Normalizing gives \eqref{eq:exact-normal-graph}. The orientation is outward
	because when $\nabla_{\partial C}d(c)=0$, formula
	\eqref{eq:exact-normal-graph} reduces to
	\[
	N_d(c)=\nu(c).
	\]
\end{proof}

An immediate consequence of \eqref{eq:exact-normal-graph} is the exact
transversality identity
\begin{equation}
N_d(c)\cdot\nu(c)
=
\frac{1}{
	\sqrt{
		1+
		\left|
		L_d(c)^{-1}\nabla_{\partial C}d(c)
		\right|^2
	}
}
>0.
\label{eq:normal-graph-angle}
\end{equation}
Thus every $C^1$ normal graph is locally transverse to the outward normal
rays of the core. This local transversality should not be confused with the
global geometric normal property \textnormal{(OC2)}, which requires the
entire inward normal ray from the outer boundary to reach the interior of
the core.
\subsection{The return construction for a fixed domain in $\mathcal O_C$}
\label{subsec:fixed-return-map}

Let
\[
\Omega\in\mathcal O_C
\]
be fixed, with background thickness function
\[
d_0\in C^{1,\alpha}(\partial C)
\]
and boundary parametrization
\[
\Psi_0(c)
=
c+d_0(c)\nu(c).
\]
For brevity, set
\[
x_0(c):=\Psi_0(c)\in\partial\Omega.
\]

The first stage of the return construction is the outward normal map
\begin{equation}
\Phi_0:\partial C\longrightarrow\partial\Omega,
\qquad
\Phi_0(c)
:=
c+d_0(c)\nu(c).
\label{eq:outward-normal-map}
\end{equation}

Let
\[
N_0(c)
:=
\nu_\Omega(\Phi_0(c))
\]
denote the outward unit normal to $\partial\Omega$ at $\Phi_0(c)$, and let
\[
n_0(c):=-N_0(c)
\]
be the corresponding inward unit normal.

By condition \textnormal{(OC2)}, the ray
\[
\Phi_0(c)+s\,n_0(c),
\qquad s\geq0,
\]
meets the interior of $C$. Since $C$ is closed and convex, the set
\[
I_c
:=
\left\{
s\geq0:
\Phi_0(c)+s\,n_0(c)\in C
\right\}
\]
is a nonempty closed interval. We define its first entrance time by
\begin{equation}
\tau_0(c)
:=
\min I_c
>0.
\label{eq:return-distance}
\end{equation}
The corresponding first contact point belongs to $\partial C$ and is
given by
\begin{equation}
\pi_0(\Phi_0(c))
:=
\Phi_0(c)+\tau_0(c)n_0(c)
\in\partial C.
\label{eq:reciprocal-map}
\end{equation}

\begin{definition}[Return map]
	\label{def:fixed-return-map}
	The return map associated with the pair $(C,\Omega)$ is
	\begin{equation}
	F_0
	:=
	\pi_0\circ\Phi_0
	:
	\partial C\longrightarrow\partial C.
	\label{eq:fixed-return-map}
	\end{equation}
	Thus, starting from $c\in\partial C$, one first follows the outward
	normal of the core until reaching $\partial\Omega$, and then follows the
	inward normal of $\partial\Omega$ until its first intersection with
	$\partial C$.
\end{definition}

This construction is intrinsic to the geometry of the pair $(C,\Omega)$.
It should not be confused with a specular-reflection law: the second leg of
the construction follows the inward normal of the outer boundary itself.

We now express the return displacement directly in terms of the thickness
function. Define
\begin{equation}
L_0(c)
:=
I-d_0(c)S_c
\label{eq:L0-definition}
\end{equation}
and
\begin{equation}
w_0(c)
:=
L_0(c)^{-1}
\nabla_{\partial C}d_0(c).
\label{eq:w0-definition}
\end{equation}
By Proposition~\ref{prop:exact-normal-graph},
\begin{equation}
N_0(c)
=
\frac{
	\nu(c)-w_0(c)
}{
	\sqrt{1+|w_0(c)|^2}
},
\label{eq:background-exact-normal}
\end{equation}
and therefore
\begin{equation}
n_0(c)
=
\frac{
	-\nu(c)+w_0(c)
}{
	\sqrt{1+|w_0(c)|^2}
}.
\label{eq:background-exact-inward-normal}
\end{equation}

Introduce the scaled return distance
\begin{equation}
\sigma_0(c)
:=
\frac{\tau_0(c)}
{\sqrt{1+|w_0(c)|^2}}.
\label{eq:sigma0-definition}
\end{equation}
Substituting
\eqref{eq:background-exact-inward-normal} into
\eqref{eq:reciprocal-map} yields the exact identity
\begin{equation}
\boxed{
	F_0(c)-c
	=
	\sigma_0(c)w_0(c)
	+
	\bigl(d_0(c)-\sigma_0(c)\bigr)\nu(c).
}
\label{eq:exact-return-displacement}
\end{equation}

Accordingly, the tangential and normal components of the return
displacement are
\begin{align}
P_{T_c\partial C}
\bigl(F_0(c)-c\bigr)
&=
\sigma_0(c)w_0(c),
\label{eq:return-tangential-component}
\\
\bigl\langle
F_0(c)-c,\nu(c)
\bigr\rangle
&=
d_0(c)-\sigma_0(c).
\label{eq:return-normal-component}
\end{align}

The identity \eqref{eq:exact-return-displacement} is exact; no
small-thickness expansion has been used.
\subsection{Return transversality and regularity of the fixed return map}
\label{subsec:return-transversality}

We now show that the first return introduced in
Subsection~\ref{subsec:fixed-return-map} is automatically transverse.
This yields the regularity needed later to compare the background return map
with the return maps generated by the PDE level sets.

For \(c\in\partial C\), write
\[
q_0(c):=F_0(c)
=
\Phi_0(c)+\tau_0(c)n_0(c)
\in\partial C.
\]

\begin{lemma}[Transversality at the return point]
	\label{lem:return-transversality}
	For every \(c\in\partial C\),
	\begin{equation}
	n_0(c)\cdot\nu(q_0(c))<0.
	\label{eq:return-pointwise-transversality}
	\end{equation}
	Equivalently, the inward normal ray issued from
	\(\Phi_0(c)\in\partial\Omega\) crosses \(\partial C\) transversely at its
	first contact point \(q_0(c)\).
\end{lemma}

\begin{proof}
	Fix \(c\in\partial C\), and set
	\[
	q=q_0(c),
	\qquad
	n=n_0(c).
	\]
	By definition of the first entrance time,
	\[
	q=\Phi_0(c)+\tau_0(c)n\in\partial C.
	\]
	Condition \textnormal{(OC2)} asserts that the same ray reaches the interior
	of \(C\). Hence there exists \(s_*>0\) such that
	\[
	q+s_* n\in\operatorname{int}C.
	\]
	
	Since \(C\) is convex and \(\nu(q)\) is the outward supporting normal at
	\(q\), every point \(y\in\operatorname{int}C\) satisfies
	\[
	(y-q)\cdot\nu(q)<0.
	\]
	Taking \(y=q+s_* n\) gives
	\[
	s_*\,n\cdot\nu(q)<0.
	\]
	Therefore
	\[
	n_0(c)\cdot\nu(q_0(c))<0.
	\]
\end{proof}

We next use this transversality to obtain smooth dependence of the return
time and return point on the initial base point.

\begin{proposition}[Regularity of the return time and return map]
	\label{prop:return-map-regularity}
	Let \(\Omega\in\mathcal O_C\). Then
	\[
	\tau_0\in C^{1,\alpha}(\partial C),
	\]
	and consequently
	\[
	F_0\in C^{1,\alpha}(\partial C,\partial C).
	\]
	Moreover, there exists a constant
	\[
	\eta_{\mathrm{ret}}>0
	\]
	such that
	\begin{equation}
	-\,n_0(c)\cdot\nu(F_0(c))
	\geq
	\eta_{\mathrm{ret}}
	\qquad
	\text{for all }c\in\partial C.
	\label{eq:return-uniform-transversality}
	\end{equation}
\end{proposition}

\begin{proof}
	Fix \(c_\ast\in\partial C\), and let
	\[
	q_\ast=F_0(c_\ast).
	\]
	Let \(\rho_C\) be the signed distance function to \(\partial C\), chosen
	negative in \(C\) and positive outside \(C\). Since
	\(\partial C\in C^{2,\alpha}\), the function \(\rho_C\) is of class
	\(C^{2,\alpha}\) in a neighborhood of \(\partial C\), with
	\[
	\nabla\rho_C(q)=\nu(q),
	\qquad q\in\partial C.
	\]
	
	Define, for \(c\) near \(c_\ast\) and \(s\) near \(\tau_0(c_\ast)\),
	\begin{equation}
	H(c,s)
	:=
	\rho_C\bigl(
	\Phi_0(c)+s\,n_0(c)
	\bigr).
	\label{eq:return-implicit-function}
	\end{equation}
	By Proposition~\ref{prop:background-thickness-regularity} and the regularity
	of \(\partial\Omega\), both \(\Phi_0\) and \(n_0\) are of class
	\(C^{1,\alpha}\). Hence \(H\in C^{1,\alpha}\).
	
	At the return point,
	\[
	H(c_\ast,\tau_0(c_\ast))=0,
	\]
	and
	\[
	\begin{aligned}
	\partial_s H(c_\ast,\tau_0(c_\ast))
	&=
	\nabla\rho_C(q_\ast)\cdot n_0(c_\ast)
	\\
	&=
	\nu(q_\ast)\cdot n_0(c_\ast)
	<0
	\end{aligned}
	\]
	by Lemma~\ref{lem:return-transversality}.
	
The implicit function theorem therefore yields a unique
\(C^{1,\alpha}\) function \(s=\tau(c)\) near \(c_\ast\) satisfying
\[
H(c,\tau(c))=0.
\]
After shrinking the neighborhood if necessary, continuity also gives
\[
\partial_sH(c,\tau(c))<0.
\]
Hence, for every nearby \(c\), the ray crosses \(\partial C\) at
\(s=\tau(c)\) from the exterior side to the interior side. Since \(C\)
is convex, its intersection with the ray
\[
\{\Phi_0(c)+s n_0(c):s\geq0\}
\]
is a closed interval. An exterior-to-interior crossing is therefore the
left endpoint of that interval, i.e. the first entrance time. Thus
\[
\tau(c)=\tau_0(c).
\]
Since \(c_\ast\) was arbitrary, the local representations patch together
and give
\[
\tau_0\in C^{1,\alpha}(\partial C).
\]

	It follows from
	\[
	F_0(c)
	=
	\Phi_0(c)+\tau_0(c)n_0(c)
	\]
	that
	\[
	F_0\in C^{1,\alpha}
	\bigl(\partial C,\partial C\bigr).
	\]
	
	Finally, the function
	\[
	c\longmapsto
	-\,n_0(c)\cdot\nu(F_0(c))
	\]
	is continuous and strictly positive on the compact manifold
	\(\partial C\). Hence it has a positive minimum, which gives
	\eqref{eq:return-uniform-transversality}.
\end{proof}

The implicit formulation also gives an exact differential identity that
will be useful in the perturbation analysis. Differentiating
\[
H(c,\tau_0(c))=0
\]
in a direction \(\tau\in T_c\partial C\) gives
\begin{equation}
\begin{split}
D\tau_0(c)[\tau]
=
-
\frac{
	\nu(F_0(c))\cdot
	\left(
	D\Phi_0(c)[\tau]
	+
	\tau_0(c)Dn_0(c)[\tau]
	\right)
}{
	\nu(F_0(c))\cdot n_0(c)
}.
\label{eq:return-time-differential}
\end{split}
\end{equation}
Consequently,
\begin{equation}
\begin{split}
DF_0(c)[\tau]
&=
D\Phi_0(c)[\tau]
+
D\tau_0(c)[\tau]\,n_0(c)
\\
&\qquad
+
\tau_0(c)Dn_0(c)[\tau].
\label{eq:return-map-differential}
\end{split}
\end{equation}

Equations \eqref{eq:return-time-differential} and
\eqref{eq:return-map-differential} make explicit the role of the return
transversality denominator
\[
\nu(F_0(c))\cdot n_0(c).
\]
The uniform bound
\eqref{eq:return-uniform-transversality}
will provide the nondegeneracy needed in the stability analysis of the
PDE-generated return maps.

\section{Elliptic problem and boundary-collar transversality}
\label{sec:pde-collar}

\subsection{The elliptic problem and the outer boundary collar}
\label{subsec:elliptic-boundary-collar}

Let
\[
\Omega=\Omega[d_0]\in\mathcal O_C
\]
be the fixed background domain introduced in Section~\ref{subsec:class-OC}.
Let \(u\) denote the unique solution of the Dirichlet problem
\begin{equation}
\begin{cases}
-\Delta u=f
& \text{in }\Omega,\\[1mm]
u=0
& \text{on }\partial\Omega,
\end{cases}
\label{eq:elliptic-problem}
\end{equation}
where
\begin{equation}
f\in C^{0,\alpha}(\overline{\Omega}),
\qquad
f\geq0,
\qquad
f\not\equiv0,
\qquad
\operatorname{supp}f\Subset\operatorname{int}C.
\label{eq:source-assumptions}
\end{equation}

By the global Schauder theory for the Dirichlet problem \cite{GilbargTrudinger2001},
\begin{equation}
u\in C^{2,\alpha}(\overline{\Omega}).
\label{eq:u-global-regularity}
\end{equation}
Since \(f\geq0\) and \(f\not\equiv0\), the strong maximum principle gives
\begin{equation}
u>0
\qquad\text{in }\Omega.
\label{eq:u-positive}
\end{equation}

Moreover, because the source is compactly supported in the interior of
\(C\),
\begin{equation}
\Delta u=0
\qquad
\text{in }\Omega\setminus\operatorname{supp}f.
\label{eq:u-harmonic-outer-region}
\end{equation}
In particular, \(u\) is harmonic in a full neighborhood of
\(\partial\Omega\).

\begin{lemma}[Uniform Hopf estimate on the outer boundary]
	\label{lem:uniform-Hopf-outer-boundary}
	There exists a constant \(m_{\Omega}>0\) such that
	\begin{equation}
	\partial_{\nu_\Omega}u(x)
	\leq
	-m_{\Omega}
	<0
	\qquad
	\text{for every }x\in\partial\Omega,
	\label{eq:uniform-Hopf-estimate}
	\end{equation}
	where \(\nu_\Omega\) denotes the outward unit normal to
	\(\partial\Omega\).
\end{lemma}

\begin{proof}
	Since
	\[
	u>0\quad\text{in }\Omega,
	\qquad
	u=0\quad\text{on }\partial\Omega,
	\]
	every boundary point is a strict boundary minimum of \(u\).
	Because \(\partial\Omega\in C^{2,\alpha}\), the interior sphere condition
	holds at every boundary point. The Hopf boundary point lemma therefore
	gives
	\[
	\partial_{\nu_\Omega}u(x)<0
	\qquad
	\text{for every }x\in\partial\Omega.
	\]
	The function
	\[
	x\longmapsto\partial_{\nu_\Omega}u(x)
	\]
	is continuous on the compact hypersurface \(\partial\Omega\). Hence its
	maximum is strictly negative, and there exists \(m_\Omega>0\) satisfying
	\eqref{eq:uniform-Hopf-estimate}.
\end{proof}

The Hopf estimate is expressed in the normal direction of
\(\partial\Omega\). For the normal-graph representation used throughout
this paper, we need the derivative along the outward normal rays of the
core.

For
\[
x=\Phi(c,r)=c+r\nu(c),
\]
define
\begin{equation}
\partial_r u(c,r)
:=
\nabla u(\Phi(c,r))\cdot\nu(c).
\label{eq:radial-derivative-definition}
\end{equation}

At the background boundary \(r=d_0(c)\), the tangential derivatives of
\(u\) along \(\partial\Omega\) vanish. Hence
\[
\nabla u(\Psi_0(c))
=
\partial_{\nu_\Omega}u(\Psi_0(c))
\,
\nu_\Omega(\Psi_0(c)).
\]
Taking the scalar product with \(\nu(c)\) gives
\begin{equation}
\begin{split}
\partial_r u(c,d_0(c))
&=
\partial_{\nu_\Omega}u(\Psi_0(c))
\,
\nu_\Omega(\Psi_0(c))\cdot\nu(c).
\end{split}
\label{eq:boundary-radial-derivative}
\end{equation}

By Lemma~\ref{lem:background-transversality},
\[
\nu_\Omega(\Psi_0(c))\cdot\nu(c)
\geq\theta_0>0.
\]
Combining this with \eqref{eq:uniform-Hopf-estimate} yields
\begin{equation}
\partial_r u(c,d_0(c))
\leq
-m_\Omega\theta_0
<0
\qquad
\text{for all }c\in\partial C.
\label{eq:uniform-boundary-radial-negativity}
\end{equation}

\begin{proposition}[Uniform radial monotonicity in an outer collar]
	\label{prop:uniform-radial-collar}
	There exist constants
	\[
	\delta_{\mathrm{col}}>0,
	\qquad
	\eta_{\mathrm{col}}>0,
	\]
	such that
	\begin{equation}
	\partial_r u(c,r)
	\leq
	-\eta_{\mathrm{col}}<0
	\label{eq:outer-collar-radial-monotonicity}
	\end{equation}
	whenever
	\[
	c\in\partial C,
	\qquad
	d_0(c)-\delta_{\mathrm{col}}
	\leq r\leq d_0(c).
	\]
\end{proposition}

\begin{proof}
	Define
	\[
	G(c,s)
	:=
	\partial_r u\bigl(c,d_0(c)-s\bigr)
	=
	\nabla u\!\left(c+\bigl(d_0(c)-s\bigr)\nu(c)\right)\cdot\nu(c).
	\]
	Since
	\[
	u\in C^{2,\alpha}(\overline{\Omega}),
	\qquad
	d_0\in C^{1,\alpha}(\partial C),
	\qquad
	\nu\in C^{1,\alpha}(\partial C),
	\]
	the function \(G\) is continuous wherever the corresponding points
	lie in \(\overline{\Omega}\).
	
	Choose first
	\[
	0<\bar\delta<\frac{d_-}{2}.
	\]
	Then
	\[
	(c,s)\longmapsto
	c+\bigl(d_0(c)-s\bigr)\nu(c)
	\]
	is well defined for
	\[
	(c,s)\in\partial C\times[0,\bar\delta],
	\]
	and its image lies in \(\overline{\Omega}\setminus C\).
	Hence \(G\) is continuous on the compact set
	\[
	\partial C\times[0,\bar\delta]
	\]
	and is therefore uniformly continuous there.
	
	At \(s=0\), equation
	\eqref{eq:uniform-boundary-radial-negativity} gives
	\[
	G(c,0)
	=
	\partial_r u(c,d_0(c))
	\leq
	-m_\Omega\theta_0
	\qquad
	\text{for every }c\in\partial C.
	\]
	By uniform continuity, there exists
	\[
	0<\delta_{\mathrm{col}}\leq\bar\delta
	\]
	such that
	\[
	\bigl|G(c,s)-G(c,0)\bigr|
	\leq
	\frac{m_\Omega\theta_0}{2}
	\]
	for every
	\[
	c\in\partial C,
	\qquad
	0\leq s\leq\delta_{\mathrm{col}}.
	\]
	Consequently,
	\[
	G(c,s)
	\leq
	-\frac{m_\Omega\theta_0}{2}.
	\]
	Returning to the variable
	\[
	r=d_0(c)-s,
	\]
	we obtain
	\[
	\partial_r u(c,r)
	\leq
	-\frac{m_\Omega\theta_0}{2}
	\]
	whenever
	\[
	d_0(c)-\delta_{\mathrm{col}}
	\leq r\leq d_0(c).
	\]
	
	Thus \eqref{eq:outer-collar-radial-monotonicity} holds with
	\[
	\boxed{
		\eta_{\mathrm{col}}
		:=
		\frac{m_\Omega\theta_0}{2}
		>0.
	}
	\]
	Moreover,
	\[
	0<\delta_{\mathrm{col}}
	\leq\bar\delta
	<
	\frac{d_-}{2},
	\]
	so the entire collar lies strictly outside the core \(C\).
\end{proof}

The result above is the only radial monotonicity statement needed in the
boundary-level analysis. No monotonicity of \(u\) along the entire normal
segment from \(\partial C\) to \(\partial\Omega\) is assumed or claimed.

\section{Superlevel boundaries as normal graphs}
\label{sec:superlevel-normal-graphs}

\subsection{Localization of small superlevel boundaries in the outer collar}
\label{subsec:small-level-localization}

For \(t>0\), define the superlevel domain
\begin{equation}
\Omega_t
:=
\{x\in\Omega:u(x)>t\}.
\label{eq:superlevel-domain}
\end{equation}

We recall from Proposition~\ref{prop:uniform-radial-collar}
that there exist constants
\[
\delta_{\mathrm{col}}>0,
\qquad
\eta_{\mathrm{col}}>0,
\]
such that
\begin{equation}
\partial_r u(c,r)
\leq
-\eta_{\mathrm{col}}<0
\label{eq:collar-radial-negativity-section4}
\end{equation}
whenever
\[
c\in\partial C,
\qquad
d_0(c)-\delta_{\mathrm{col}}
\leq r\leq d_0(c).
\]
We assume, as in Section~\ref{subsec:elliptic-boundary-collar}, that
\[
0<\delta_{\mathrm{col}}<\frac{d_-}{2}.
\]

Define the closed inner region
\begin{equation}
\begin{split}
K_{\mathrm{col}}
:=
C
\cup
\bigl\{
\Phi(c,r):
c\in\partial C,\;
0\leq r\leq d_0(c)-\delta_{\mathrm{col}}
\bigr\}.
\label{eq:inner-region-section4}
\end{split}
\end{equation}
Since \(d_0\) is continuous and
\[
d_0(c)-\delta_{\mathrm{col}}>0,
\]
the set \(K_{\mathrm{col}}\) is compact. Moreover,
\[
K_{\mathrm{col}}\Subset\Omega.
\]

Because \(u>0\) in \(\Omega\), we may therefore define
\begin{equation}
m_{\mathrm{col}}
:=
\min_{K_{\mathrm{col}}}u
>0.
\label{eq:mcol-definition}
\end{equation}

Fix
\begin{equation}
0<t_{\mathrm{col}}<m_{\mathrm{col}}.
\label{eq:tcol-definition}
\end{equation}

\begin{proposition}[Localization of small superlevel boundaries]
	\label{prop:small-level-localization}
	For every
	\[
	0<t<t_{\mathrm{col}},
	\]
	one has
	\begin{equation}
	\{x\in\Omega:u(x)=t\}
	\subset
	\mathcal A_{\delta_{\mathrm{col}}},
	\label{eq:small-level-in-collar}
	\end{equation}
	where
	\begin{equation}
	\mathcal A_{\delta_{\mathrm{col}}}
	:=
	\left\{
	\Phi(c,r):
	c\in\partial C,\;
	d_0(c)-\delta_{\mathrm{col}}
	<r<d_0(c)
	\right\}.
	\label{eq:outer-collar-section4}
	\end{equation}
	
	Moreover, every \(t\in(0,t_{\mathrm{col}})\) is a regular value of \(u\).
	Consequently,
	\begin{equation}
	\partial\Omega_t
	=
	\{x\in\Omega:u(x)=t\}
	\label{eq:boundary-level-set}
	\end{equation}
	and \(\partial\Omega_t\) is a \(C^{2,\alpha}\) hypersurface.
\end{proposition}

\begin{proof}
	Let
	\[
	0<t<t_{\mathrm{col}}.
	\]
	By definition of \(m_{\mathrm{col}}\),
	\[
	u(x)\geq m_{\mathrm{col}}>t
	\qquad
	\text{for every }x\in K_{\mathrm{col}}.
	\]
	Hence
	\[
	\{u=t\}\cap K_{\mathrm{col}}
	=
	\varnothing.
	\]
	
	Now let \(x\in\Omega\setminus K_{\mathrm{col}}\).
	Since \(C\subset K_{\mathrm{col}}\), we have \(x\in\Omega\setminus C\).
	By the global normal-coordinate representation of
	\(\Omega\in\mathcal O_C\), there exists a unique pair
	\[
	(c,r)\in\partial C\times(0,\infty)
	\]
	such that
	\[
	x=\Phi(c,r),
	\qquad
	0<r<d_0(c).
	\]
	The condition \(x\notin K_{\mathrm{col}}\) implies
	\[
	r>d_0(c)-\delta_{\mathrm{col}}.
	\]
	Thus
	\[
	d_0(c)-\delta_{\mathrm{col}}
	<r<d_0(c),
	\]
	and therefore
	\[
	x\in\mathcal A_{\delta_{\mathrm{col}}}.
	\]
	This proves \eqref{eq:small-level-in-collar}.
	
	It remains to prove regularity of the level.
	If \(x=\Phi(c,r)\in\{u=t\}\), then by
	\eqref{eq:small-level-in-collar} and
	\eqref{eq:collar-radial-negativity-section4},
	\[
	\nabla u(x)\cdot\nu(c)
	=
	\partial_r u(c,r)
	\leq
	-\eta_{\mathrm{col}}<0.
	\]
	Hence
	\[
	\nabla u(x)\neq0.
	\]
	Thus every \(t\in(0,t_{\mathrm{col}})\) is a regular value.
	
	Since
	\[
	\Omega_t=\{u>t\},
	\]
	the regular-value theorem gives
	\[
	\partial\Omega_t=\{u=t\}.
	\]
	Finally, because
	\[
	u\in C^{2,\alpha}(\overline{\Omega}),
	\]
	the implicit function theorem yields
	\[
	\partial\Omega_t\in C^{2,\alpha}.
	\]
\end{proof}

\begin{remark}[Boundary-layer character of the small-\(t\) regime]
	\label{rem:boundary-layer-regime}
	The localization statement is a boundary-layer result: as
	\(t\downarrow0\), the relevant level hypersurfaces lie near the fixed outer
	boundary \(\partial\Omega\), not near the core \(\partial C\).
	
	No estimate based on decay away from the source is required. The argument
	uses only positivity of \(u\) on a compact subset of the interior and the
	uniform collar transversality established in Section~3.
\end{remark}

\subsection{Global normal-graph representation of the superlevel boundary}
\label{subsec:global-superlevel-normal-graph}

We now use the localization result of
Proposition~\ref{prop:small-level-localization} together with the strict
radial monotonicity in the outer collar to represent the entire superlevel
boundary as a single-valued normal graph over \(\partial C\).

Fix
\[
0<t<t_{\mathrm{col}}.
\]
For each \(c\in\partial C\), consider the one-dimensional function
\begin{equation}
g_c(r)
:=
u\bigl(\Phi(c,r)\bigr)
=
u(c+r\nu(c)),
\qquad
d_0(c)-\delta_{\mathrm{col}}
\leq r\leq d_0(c).
\label{eq:gc-definition}
\end{equation}
By construction,
\begin{equation}
g_c(d_0(c))=0,
\label{eq:gc-outer-value}
\end{equation}
whereas
\begin{equation}
g_c\bigl(d_0(c)-\delta_{\mathrm{col}}\bigr)
\geq
m_{\mathrm{col}}
>
t.
\label{eq:gc-inner-value}
\end{equation}
Moreover,
\begin{equation}
g_c'(r)
=
\partial_r u(c,r)
\leq
-\eta_{\mathrm{col}}<0
\label{eq:gc-strict-monotonicity}
\end{equation}
throughout the interval.

\begin{theorem}[Global normal-graph representation]
	\label{thm:global-superlevel-normal-graph}
	For every
	\[
	0<t<t_{\mathrm{col}}
	\]
	there exists a unique function
	\[
	d_t:\partial C\longrightarrow(0,\infty)
	\]
	such that
	\begin{equation}
	u\bigl(c+d_t(c)\nu(c)\bigr)=t
	\qquad
	\text{for every }c\in\partial C.
	\label{eq:dt-level-equation}
	\end{equation}
	It satisfies
	\begin{equation}
	d_0(c)-\delta_{\mathrm{col}}
	<
	d_t(c)
	<
	d_0(c)
	\qquad
	\text{for every }c\in\partial C,
	\label{eq:dt-collar-bounds}
	\end{equation}
	and
	\begin{equation}
	d_t\in C^{1,\alpha}(\partial C).
	\label{eq:dt-regularity}
	\end{equation}
	
	Furthermore,
	\begin{equation}
	\boxed{
		\partial\Omega_t
		=
		\Gamma_{d_t}
		=
		\left\{
		c+d_t(c)\nu(c):
		c\in\partial C
		\right\}.
	}
	\label{eq:superlevel-boundary-normal-graph}
	\end{equation}
\end{theorem}

\begin{proof}
	Fix \(c\in\partial C\). By
	\eqref{eq:gc-outer-value} and
	\eqref{eq:gc-inner-value},
	\[
	g_c(d_0(c))<t
	<
	g_c(d_0(c)-\delta_{\mathrm{col}}).
	\]
	The intermediate value theorem therefore yields at least one number
	\[
	d_t(c)
	\in
	\bigl(
	d_0(c)-\delta_{\mathrm{col}},
	d_0(c)
	\bigr)
	\]
	such that
	\[
	g_c(d_t(c))=t.
	\]
	By the strict monotonicity
	\eqref{eq:gc-strict-monotonicity}, this value is unique. Thus
	\(d_t\) is well defined on all of \(\partial C\).
	
	To prove regularity, define
	\begin{equation}
	U(c,r)
	:=
	u(\Phi(c,r))
	=
	u(c+r\nu(c)).
	\label{eq:U-section4}
	\end{equation}
	Then
	\[
	U(c,d_t(c))=t
	\]
	and
	\[
	\partial_rU(c,d_t(c))
	=
	\partial_r u(c,d_t(c))
	\leq
	-\eta_{\mathrm{col}}<0.
	\]
	Since \(u\in C^{2,\alpha}(\overline{\Omega})\) and
	\(\partial C\in C^{2,\alpha}\), the map \(U\) is of class
	\(C^{1,\alpha}\) in the variables \((c,r)\) on the outer collar.
	The implicit function theorem therefore gives
	\[
	d_t\in C^{1,\alpha}(\partial C).
	\]
	
	It remains to identify the complete level boundary.
	By Proposition~\ref{prop:small-level-localization},
	every point of \(\{u=t\}\) belongs to the outer collar. Hence it possesses
	a unique normal-coordinate representation
	\[
	x=\Phi(c,r)
	\]
	with
	\[
	d_0(c)-\delta_{\mathrm{col}}
	<
	r
	<
	d_0(c).
	\]
	Since
	\[
	u(\Phi(c,r))=t,
	\]
	the uniqueness established above forces
	\[
	r=d_t(c).
	\]
	Thus
	\[
	\{u=t\}
	\subset
	\Gamma_{d_t}.
	\]
	The reverse inclusion follows immediately from
	\eqref{eq:dt-level-equation}. Hence
	\[
	\{u=t\}
	=
	\Gamma_{d_t}.
	\]
	Finally, Proposition~\ref{prop:small-level-localization} gives
	\[
	\partial\Omega_t=\{u=t\},
	\]
	which proves
	\eqref{eq:superlevel-boundary-normal-graph}.
\end{proof}

The graph representation also determines the full superlevel domain.

\begin{proposition}[Superlevel domain generated by \(d_t\)]
	\label{prop:superlevel-domain-generated-by-dt}
	For every
	\[
	0<t<t_{\mathrm{col}},
	\]
	one has
	\begin{equation}
	\boxed{
		\Omega_t
		=
		\Omega[d_t]
		=
		C
		\cup
		\left\{
		\Phi(c,r):
		c\in\partial C,\;
		0<r<d_t(c)
		\right\}.
	}
	\label{eq:Omega-t-generated-by-dt}
	\end{equation}
\end{proposition}

\begin{proof}
	Let \(c\in\partial C\). In the outer collar, the function
	\[
	r\longmapsto U(c,r)
	\]
	is strictly decreasing and satisfies
	\[
	U(c,d_t(c))=t.
	\]
	Hence
	\begin{equation}
	U(c,r)>t
	\qquad
	\text{whenever }
	d_0(c)-\delta_{\mathrm{col}}
	\leq r<d_t(c),
	\label{eq:outer-ray-superlevel}
	\end{equation}
	whereas
	\[
	U(c,r)<t
	\qquad
	\text{whenever }
	d_t(c)<r\leq d_0(c).
	\]
	
	On the compact inner region \(K_{\mathrm{col}}\), we have
	\[
	u\geq m_{\mathrm{col}}>t.
	\]
	Therefore every point on the same normal ray with
	\[
	0\leq r\leq d_0(c)-\delta_{\mathrm{col}}
	\]
	also belongs to \(\Omega_t\).
	
	Combining the inner region with
	\eqref{eq:outer-ray-superlevel} gives
	\[
	u(\Phi(c,r))>t
	\quad\Longleftrightarrow\quad
	0<r<d_t(c).
	\]
	Since \(u>t\) on \(C\), the representation
	\eqref{eq:Omega-t-generated-by-dt} follows.
\end{proof}

\begin{remark}[What has been established]
	\label{rem:graph-representation-scope}
	The preceding results establish the global graph structure of the small
	superlevel domains but make no asymptotic claim about \(d_t\) as
	\(t\downarrow0\).
	
	The asymptotic comparison between \(d_t\) and the background thickness
	\(d_0\) is treated separately in the next subsection.
\end{remark}

\subsection{The boundary-layer variable and quantitative convergence to $d_0$}
\label{subsec:boundary-layer-variable}

The normal-graph representation obtained in
Theorem~\ref{thm:global-superlevel-normal-graph} allows us to measure the
motion of the superlevel boundary relative to the fixed outer boundary.

\begin{definition}[Boundary-layer thickness]
	\label{def:boundary-layer-thickness}
	For
	\[
	0<t<t_{\mathrm{col}},
	\]
	define
	\begin{equation}
	h_t(c)
	:=
	d_0(c)-d_t(c),
	\qquad c\in\partial C.
	\label{eq:ht-definition}
	\end{equation}
\end{definition}

By \eqref{eq:dt-collar-bounds},
\begin{equation}
0<h_t(c)<\delta_{\mathrm{col}}
\qquad
\text{for every }c\in\partial C.
\label{eq:ht-positive}
\end{equation}

Thus \(h_t\), rather than \(d_t\), is the small geometric quantity in the
limit \(t\downarrow0\).

Since
\[
U(c,d_t(c))=t
\]
and
\[
U(c,d_0(c))=0,
\]
the fundamental theorem of calculus gives the exact identity
\begin{equation}
\begin{split}
t
&=
U(c,d_t(c))-U(c,d_0(c))
\\
&=
-\int_{d_t(c)}^{d_0(c)}
\partial_rU(c,r)\,dr.
\end{split}
\label{eq:exact-ht-integral-identity}
\end{equation}

The collar estimate provides a uniform lower bound for the integrand.
For the complementary upper bound, set
\begin{equation}
M_{\mathrm{col}}
:=
\max_{\overline{\mathcal A_{\delta_{\mathrm{col}}}}}
|\partial_r u|
<\infty.
\label{eq:Mcol-definition}
\end{equation}
Hence, throughout the closed outer collar,
\begin{equation}
\eta_{\mathrm{col}}
\leq
-\partial_rU(c,r)
\leq
M_{\mathrm{col}}.
\label{eq:radial-derivative-two-sided}
\end{equation}

\begin{proposition}[Two-sided boundary-layer estimate]
	\label{prop:boundary-layer-two-sided}
	For every
	\[
	0<t<t_{\mathrm{col}}
	\]
	and every \(c\in\partial C\),
	\begin{equation}
	\boxed{
		\frac{t}{M_{\mathrm{col}}}
		\leq
		h_t(c)
		\leq
		\frac{t}{\eta_{\mathrm{col}}}.
	}
	\label{eq:ht-two-sided}
	\end{equation}
	Consequently,
	\begin{equation}
	\|d_t-d_0\|_{L^\infty(\partial C)}
	=
	\|h_t\|_{L^\infty(\partial C)}
	\leq
	\frac{t}{\eta_{\mathrm{col}}}.
	\label{eq:dt-C0-rate}
	\end{equation}
	In particular,
	\begin{equation}
	\boxed{
		d_t\longrightarrow d_0
		\qquad
		\text{uniformly on }\partial C
		\quad\text{as }t\downarrow0.
	}
	\label{eq:dt-C0-convergence}
	\end{equation}
\end{proposition}

\begin{proof}
	From \eqref{eq:exact-ht-integral-identity} and
	\eqref{eq:radial-derivative-two-sided},
	\[
	\eta_{\mathrm{col}}
	\int_{d_t(c)}^{d_0(c)}dr
	\leq
	t
	\leq
	M_{\mathrm{col}}
	\int_{d_t(c)}^{d_0(c)}dr.
	\]
	Since
	\[
	d_0(c)-d_t(c)=h_t(c),
	\]
	we obtain
	\[
	\eta_{\mathrm{col}}h_t(c)
	\leq
	t
	\leq
	M_{\mathrm{col}}h_t(c).
	\]
	Rearranging yields
	\eqref{eq:ht-two-sided}.
	
	The estimate
	\eqref{eq:dt-C0-rate} follows immediately from
	\[
	d_t-d_0=-h_t,
	\]
	and the uniform convergence
	\eqref{eq:dt-C0-convergence} follows by letting \(t\downarrow0\).
\end{proof}

\begin{corollary}[Linear boundary-layer scale]
	\label{cor:boundary-layer-scale}
	The boundary-layer thickness is uniformly of order \(t\):
	\begin{equation}
	h_t\asymp t
	\qquad
	\text{uniformly on }\partial C.
	\label{eq:ht-asymp-t}
	\end{equation}
	More precisely,
	\[
	\frac{1}{M_{\mathrm{col}}}
	\leq
	\frac{h_t(c)}{t}
	\leq
	\frac{1}{\eta_{\mathrm{col}}}
	\]
	for every
	\[
	c\in\partial C,
	\qquad
	0<t<t_{\mathrm{col}}.
	\]
\end{corollary}

\begin{remark}[Correct asymptotic variable]
	\label{rem:correct-boundary-layer-variable}
	The estimate \eqref{eq:ht-two-sided} makes explicit the asymptotic regime
	underlying the remainder of the paper:
	\[
	h_t=d_0-d_t\longrightarrow0,
	\]
	whereas
	\[
	d_t\longrightarrow d_0.
	\]
	Thus the moving superlevel boundary approaches the fixed outer boundary
	\(\partial\Omega=\Gamma_{d_0}\), not the core boundary \(\partial C\).
	
	In particular, any perturbative expansion of the moving geometry must be
	performed around \(d_0\), with \(h_t\) as the small quantity.
\end{remark}
\subsection{$C^1$ convergence and first-order boundary displacement}
\label{subsec:C1-boundary-displacement}

We now strengthen the uniform convergence obtained in
Proposition~\ref{prop:boundary-layer-two-sided} and determine the leading
profile of the boundary-layer thickness.

Recall the pull-back
\[
U(c,r)
:=
u\bigl(\Phi(c,r)\bigr)
=
u(c+r\nu(c)).
\]
For fixed \(r\), let
\[
\nabla_{\partial C}U(c,r)
\]
denote the intrinsic gradient of the function
\(c\mapsto U(c,r)\) on the fixed manifold \(\partial C\). Set
\begin{equation}
A(c,r)
:=
\nabla_{\partial C}U(c,r),
\qquad
B(c,r)
:=
\partial_rU(c,r).
\label{eq:AB-definition-section4}
\end{equation}

For every \(\tau\in T_c\partial C\), the identity
\[
D_c\Phi(c,r)[\tau]
=
(I-rS_c)\tau
\]
gives
\[
D_cU(c,r)[\tau]
=
\nabla u(\Phi(c,r))
\cdot
(I-rS_c)\tau.
\]
Hence
\begin{equation}
A(c,r)
=
(I-rS_c)
P_c\nabla u(\Phi(c,r)),
\label{eq:A-geometric-form}
\end{equation}
where
\[
P_c
=
I-\nu(c)\otimes\nu(c)
\]
is the orthogonal projection onto \(T_c\partial C\). Moreover,
\begin{equation}
B(c,r)
=
\nabla u(\Phi(c,r))\cdot\nu(c).
\label{eq:B-geometric-form}
\end{equation}

Differentiating
\[
U(c,d_t(c))=t
\]
along \(\partial C\) yields
\begin{equation}
A(c,d_t(c))
+
B(c,d_t(c))
\nabla_{\partial C}d_t(c)
=
0.
\label{eq:dt-gradient-chain-rule}
\end{equation}
Since
\[
B(c,d_t(c))
\leq
-\eta_{\mathrm{col}}<0,
\]
we obtain the exact identity
\begin{equation}
\boxed{
	\nabla_{\partial C}d_t(c)
	=
	-
	\frac{A(c,d_t(c))}
	{B(c,d_t(c))}.
}
\label{eq:exact-dt-gradient-section4}
\end{equation}

Similarly, differentiating
\[
U(c,d_0(c))=0
\]
gives
\begin{equation}
\boxed{
	\nabla_{\partial C}d_0(c)
	=
	-
	\frac{A(c,d_0(c))}
	{B(c,d_0(c))}.
}
\label{eq:exact-d0-gradient-section4}
\end{equation}

\begin{theorem}[$C^1$ convergence of the thickness profile]
	\label{thm:dt-C1-rate-section4}
	There exists a constant \(C_{d}>0\) such that
	\begin{equation}
	\boxed{
		\|d_t-d_0\|_{C^1(\partial C)}
		\leq
		C_{d}\,t
	}
	\label{eq:dt-C1-linear-rate}
	\end{equation}
	for every sufficiently small \(t>0\). In particular,
	\begin{equation}
	d_t\longrightarrow d_0
	\qquad
	\text{in }C^1(\partial C)
	\quad\text{as }t\downarrow0.
	\label{eq:dt-C1-convergence-section4}
	\end{equation}
\end{theorem}

\begin{proof}
	The \(L^\infty\) estimate has already been proved:
	\[
	\|d_t-d_0\|_{L^\infty(\partial C)}
	\leq
	\frac{t}{\eta_{\mathrm{col}}}.
	\]
	
	It remains to estimate the tangential gradients.
	
	The regularity
	\[
	u\in C^{2,\alpha}(\overline{\Omega}),
	\qquad
	\partial C\in C^{2,\alpha},
	\]
	together with the explicit formulas
	\eqref{eq:A-geometric-form}--\eqref{eq:B-geometric-form},
	implies that \(A\) and \(B\) are uniformly Lipschitz with respect to the
	radial variable on the closed outer collar. Thus there exists
	\(L_{\mathrm{col}}>0\) such that
	\begin{align}
	|A(c,r)-A(c,s)|
	&\leq
	L_{\mathrm{col}}|r-s|,
	\label{eq:A-r-Lipschitz}
	\\
	|B(c,r)-B(c,s)|
	&\leq
	L_{\mathrm{col}}|r-s|
	\label{eq:B-r-Lipschitz}
	\end{align}
	for every \(c\in\partial C\) and all admissible \(r,s\) in the collar.
	
	Furthermore,
	\[
	|B(c,r)|
	\geq
	\eta_{\mathrm{col}}
	\]
	throughout the collar, and \(A\) is uniformly bounded there.
	
	Using
	\eqref{eq:exact-dt-gradient-section4} and
	\eqref{eq:exact-d0-gradient-section4}, we write
	\begin{align*}
	\nabla_{\partial C}d_t
	-
	\nabla_{\partial C}d_0
	&=
	-
	\frac{A(c,d_t)}{B(c,d_t)}
	+
	\frac{A(c,d_0)}{B(c,d_0)}
	\\
	&=
	-
	\frac{A(c,d_t)-A(c,d_0)}
	{B(c,d_t)}
	\\
	&\quad
	+
	A(c,d_0)
	\frac{B(c,d_t)-B(c,d_0)}
	{B(c,d_t)B(c,d_0)}.
	\end{align*}
	Consequently,
	\[
	\left|
	\nabla_{\partial C}d_t
	-
	\nabla_{\partial C}d_0
	\right|
	\leq
	C
	|d_t-d_0|
	\]
	for a constant \(C>0\) independent of \(t\). Therefore,
	\[
	\|
	\nabla_{\partial C}d_t
	-
	\nabla_{\partial C}d_0
	\|_{L^\infty(\partial C)}
	\leq
	C
	\|d_t-d_0\|_{L^\infty(\partial C)}
	\leq
	\frac{C}{\eta_{\mathrm{col}}}\,t.
	\]
	Combining this estimate with the \(L^\infty\) bound proves
	\eqref{eq:dt-C1-linear-rate}.
\end{proof}

We now identify the leading-order profile of the moving boundary.

\begin{definition}[Boundary-speed coefficient]
	\label{def:boundary-speed-coefficient}
	Define
	\begin{equation}
	a_0(c)
	:=
	-B(c,d_0(c))
	=
	-\partial_rU(c,d_0(c)).
	\label{eq:a0-section4}
	\end{equation}
	By the collar transversality estimate,
	\begin{equation}
	a_0(c)
	\geq
	\eta_{\mathrm{col}}>0.
	\label{eq:a0-positive-section4}
	\end{equation}
	We also set
	\begin{equation}
	b_0(c)
	:=
	\frac{1}{a_0(c)}.
	\label{eq:b0-section4}
	\end{equation}
Because \(B(c,r)=\partial_rU(c,r)\) is \(C^{1,\alpha}\) in the fixed
collar variables and \(d_0\in C^{1,\alpha}(\partial C)\), one has
\[
a_0,b_0\in C^{1,\alpha}(\partial C).
\]
\end{definition}

\begin{proposition}[First-order boundary-layer profile]
	\label{prop:first-order-boundary-layer-profile}
	As \(t\downarrow0\),
	\begin{equation}
	\boxed{
		h_t(c)
		=
		t\,b_0(c)+O(t^2)
	}
	\label{eq:ht-first-order-C0}
	\end{equation}
	uniformly on \(\partial C\). Equivalently,
	\begin{equation}
	d_t(c)
	=
	d_0(c)
	-
	\frac{t}{a_0(c)}
	+
	O(t^2)
	\label{eq:dt-first-order-C0}
	\end{equation}
	uniformly in \(c\).
	
	In particular,
	\begin{equation}
	\boxed{
		\frac{h_t}{t}
		\longrightarrow
		b_0
		=
		\frac1{a_0}
		\qquad
		\text{uniformly on }\partial C.
	}
	\label{eq:ht-over-t-limit}
	\end{equation}
\end{proposition}

\begin{proof}
	From the exact identity
	\eqref{eq:exact-ht-integral-identity},
	\[
	t
	=
	\int_0^{h_t(c)}
	\bigl[
	-B(c,d_0(c)-s)
	\bigr]\,ds.
	\]
	Since \(B\) is uniformly Lipschitz in the radial variable,
	\[
	-B(c,d_0(c)-s)
	=
	a_0(c)+O(s)
	\]
	uniformly with respect to \(c\). Hence
	\[
	t
	=
	a_0(c)h_t(c)
	+
	O(h_t(c)^2).
	\]
	By Proposition~\ref{prop:boundary-layer-two-sided},
	\[
	h_t=O(t)
	\]
	uniformly, and therefore
	\[
	t
	=
	a_0(c)h_t(c)+O(t^2).
	\]
	Since \(a_0\) is uniformly bounded away from zero,
	\[
	h_t(c)
	=
	\frac{t}{a_0(c)}+O(t^2),
	\]
	which proves the claim.
\end{proof}

The first-order expansion also holds in the topology needed for the moving
normal geometry.

\begin{proposition}[$C^1$ first variation of the thickness]
	\label{prop:dt-first-variation-C1}
	As \(t\downarrow0\),
	\begin{equation}
	\boxed{
		d_t
		=
		d_0-tb_0+o_{C^1}(t).
	}
	\label{eq:dt-first-order-C1}
	\end{equation}
	Equivalently,
	\begin{equation}
	\frac{h_t}{t}
	\longrightarrow
	b_0
	\qquad
	\text{in }C^1(\partial C).
	\label{eq:ht-over-t-C1}
	\end{equation}
\end{proposition}

\begin{proof}
	The \(C^0\) convergence follows from
	Proposition~\ref{prop:first-order-boundary-layer-profile}. It remains to
	identify the limit of the tangential derivatives.
	
	Define
	\begin{equation}
	Q(c,r)
	:=
	-\frac{A(c,r)}{B(c,r)}.
	\label{eq:Q-section4}
	\end{equation}
	Then
	\[
	\nabla_{\partial C}d_t(c)
	=
	Q(c,d_t(c)),
	\qquad
	\nabla_{\partial C}d_0(c)
	=
	Q(c,d_0(c)).
	\]
	The regularity of \(u\) implies that \(\partial_rQ\) exists and is
	continuous on the closed outer collar. Therefore,
	\begin{equation}
	\begin{split}
	\frac{
		\nabla_{\partial C}d_t(c)
		-
		\nabla_{\partial C}d_0(c)
	}{t}
	&=
	\left[
	\int_0^1
	\partial_rQ
	\bigl(
	c,
	d_0(c)+s(d_t(c)-d_0(c))
	\bigr)\,ds
	\right]
	\frac{d_t(c)-d_0(c)}{t}.
	\end{split}
	\label{eq:gradient-difference-quotient-section4}
	\end{equation}
	Using
	\[
	\frac{d_t-d_0}{t}
	\longrightarrow
	-b_0
	\]
	uniformly, we obtain
	\begin{equation}
	\frac{
		\nabla_{\partial C}d_t
		-
		\nabla_{\partial C}d_0
	}{t}
	\longrightarrow
	-b_0\,
	\partial_rQ(c,d_0(c))
	\label{eq:gradient-first-limit-preid}
	\end{equation}
	uniformly on \(\partial C\).
	
	We now identify this limit. Since
	\[
	Q(c,d_0(c))
	=
	\nabla_{\partial C}d_0(c),
	\]
	equation
	\eqref{eq:exact-d0-gradient-section4} gives
	\[
	A(c,d_0)
	=
	-B(c,d_0)
	\nabla_{\partial C}d_0.
	\]
	All tangential gradients in this subsection are taken with respect to
the fixed Riemannian metric on \(\partial C\). In fixed local coordinates
on \(\partial C\), the regularity of \(U\) permits the commutation
\[
\partial_r(\partial_iU)=\partial_i(\partial_rU).
\]
Consequently, because \(A=\nabla_{\partial C}U\) and
\(B=\partial_rU\),
\begin{equation}
\partial_rA
=
\nabla_{\partial C}B.
\label{eq:mixed-derivative-identity}
\end{equation}
	Hence
	\[
	\partial_rQ(c,d_0)
	=
	-
	\frac{
		\partial_rA(c,d_0)
		+
		\partial_rB(c,d_0)
		\nabla_{\partial C}d_0
	}{
		B(c,d_0)
	}.
	\]
	
	On the other hand,
	\[
	b_0(c)
	=
	-\frac{1}{B(c,d_0(c))}.
	\]
	Differentiating this identity along \(\partial C\) yields
	\[
	\nabla_{\partial C}b_0
	=
	\frac{
		\partial_rA(c,d_0)
		+
		\partial_rB(c,d_0)
		\nabla_{\partial C}d_0
	}{
		B(c,d_0)^2
	}.
	\]
	Consequently,
	\begin{equation}
	b_0(c)\,
	\partial_rQ(c,d_0(c))
	=
	\nabla_{\partial C}b_0(c).
	\label{eq:b0-Q-identity}
	\end{equation}
	Combining
	\eqref{eq:gradient-first-limit-preid} and
	\eqref{eq:b0-Q-identity} gives
	\[
	\frac{
		\nabla_{\partial C}d_t
		-
		\nabla_{\partial C}d_0
	}{t}
	\longrightarrow
	-\nabla_{\partial C}b_0
	\]
	uniformly. Together with the \(C^0\) convergence, this proves
	\eqref{eq:dt-first-order-C1}.
\end{proof}

\begin{remark}[Geometric meaning of the boundary speed]
	\label{rem:boundary-speed-geometric-meaning}
	At the fixed outer boundary,
	\[
	\nabla u
	=
	\partial_{\nu_\Omega}u\,\nu_\Omega.
	\]
	Consequently,
	\begin{equation}
	a_0(c)
	=
	-
	\partial_{\nu_\Omega}u(\Psi_0(c))
	\,
	\nu_\Omega(\Psi_0(c))\cdot\nu(c).
	\label{eq:a0-Hopf-angle}
	\end{equation}
	Thus the leading normal-coordinate speed
	\[
	b_0(c)=\frac1{a_0(c)}
	\]
	contains both the elliptic Hopf derivative and the geometric
	transversality angle between the fixed outer boundary and the normal rays
	of the core.
\end{remark}


\section{Exact moving-boundary geometry}
\label{sec:exact-moving-boundary-geometry}

\subsection{Exact differential and induced metric of the moving boundary}
\label{subsec:exact-moving-boundary-metric}

For every sufficiently small \(t>0\), Section~\ref{sec:superlevel-normal-graphs}
gives
\[
\partial\Omega_t
=
\Gamma_{d_t}
=
\left\{
c+d_t(c)\nu(c):
c\in\partial C
\right\},
\]
with
\[
d_t\in C^{1,\alpha}(\partial C).
\]
We write
\begin{equation}
\Psi_t(c)
:=
c+d_t(c)\nu(c),
\qquad c\in\partial C,
\label{eq:Psi-t-definition}
\end{equation}
and introduce the tangent-space endomorphism
\begin{equation}
L_t(c)
:=
I-d_t(c)S_c.
\label{eq:Lt-definition}
\end{equation}

Recall that our convention is
\[
S_c=-D\nu(c)\big|_{T_c\partial C}.
\]
Accordingly, the eigenvalues of \(L_t(c)\) in a principal frame are
\[
1+d_t(c)\kappa_i(c),
\qquad i=1,\ldots,N-1,
\]
and therefore
\begin{equation}
L_t(c)>0.
\label{eq:Lt-positive}
\end{equation}

\begin{proposition}[Exact differential of the moving graph]
	\label{prop:exact-Psi-t-differential}
	For every \(c\in\partial C\) and
	\(\tau\in T_c\partial C\),
	\begin{equation}
	\boxed{
		D\Psi_t(c)[\tau]
		=
		L_t(c)\tau
		+
		\bigl(
		\nabla_{\partial C}d_t(c)\cdot\tau
		\bigr)\nu(c).
	}
	\label{eq:exact-Psi-t-differential}
	\end{equation}
	In particular, \(D\Psi_t(c)\) is injective for every \(c\).
\end{proposition}

\begin{proof}
	Differentiating
	\[
	\Psi_t(c)
	=
	c+d_t(c)\nu(c)
	\]
	in the direction \(\tau\in T_c\partial C\) gives
	\[
	\begin{aligned}
	D\Psi_t(c)[\tau]
	&=
	\tau
	+
	\bigl(
	\nabla_{\partial C}d_t(c)\cdot\tau
	\bigr)\nu(c)
	+
	d_t(c)D\nu(c)[\tau].
	\end{aligned}
	\]
	Since
	\[
	D\nu(c)[\tau]
	=
	-S_c\tau,
	\]
	we obtain
	\[
	D\Psi_t(c)[\tau]
	=
	\bigl(I-d_t(c)S_c\bigr)\tau
	+
	\bigl(
	\nabla_{\partial C}d_t(c)\cdot\tau
	\bigr)\nu(c),
	\]
	which is \eqref{eq:exact-Psi-t-differential}.
	
	If
	\[
	D\Psi_t(c)[\tau]=0,
	\]
	then taking the tangential component relative to
	\(T_c\partial C\oplus\mathbb R\nu(c)\) gives
	\[
	L_t(c)\tau=0.
	\]
	Since \(L_t(c)\) is invertible, \(\tau=0\).
\end{proof}

Let
\[
\gamma
\]
denote the Riemannian metric on \(\partial C\) induced by the Euclidean
metric of \(\mathbb R^N\). We denote by
\[
\gamma_t:=\Psi_t^\ast g_{\mathbb R^N}
\]
the pull-back of the Euclidean metric from \(\partial\Omega_t\) to
\(\partial C\).

\begin{proposition}[Exact induced metric]
	\label{prop:exact-moving-metric}
	For every
	\[
	\tau,\sigma\in T_c\partial C,
	\]
	the induced metric satisfies
	\begin{equation}
	\boxed{
		\begin{aligned}
		\gamma_t(\tau,\sigma)
		&=
		\gamma\bigl(
		L_t\tau,
		L_t\sigma
		\bigr)
		\\
		&\quad+
		\bigl(
		\nabla_{\partial C}d_t\cdot\tau
		\bigr)
		\bigl(
		\nabla_{\partial C}d_t\cdot\sigma
		\bigr).
		\end{aligned}
	}
	\label{eq:exact-moving-metric}
	\end{equation}
	
	Equivalently, as a bilinear form on \(T_c\partial C\),
	\begin{equation}
	\boxed{
		\gamma_t
		=
		L_t^2
		+
		\nabla_{\partial C}d_t
		\otimes
		\nabla_{\partial C}d_t.
	}
	\label{eq:metric-operator-form}
	\end{equation}
\end{proposition}

\begin{proof}
	Using
	\eqref{eq:exact-Psi-t-differential},
	\[
	\begin{aligned}
	\gamma_t(\tau,\sigma)
	&=
	D\Psi_t[\tau]\cdot D\Psi_t[\sigma]
	\\
	&=
	L_t\tau\cdot L_t\sigma
	\\
	&\quad+
	\bigl(
	\nabla_{\partial C}d_t\cdot\tau
	\bigr)
	\bigl(
	\nabla_{\partial C}d_t\cdot\sigma
	\bigr),
	\end{aligned}
	\]
	because
	\[
	L_t\tau,\ L_t\sigma\in T_c\partial C
	\]
	are orthogonal to \(\nu(c)\).
	Since \(S_c\), and hence \(L_t(c)\), is symmetric with respect to
	\(\gamma\), this is equivalent to
	\eqref{eq:metric-operator-form}.
\end{proof}

In a \(\gamma\)-orthonormal principal frame
\[
\{e_1,\ldots,e_{N-1}\}
\]
satisfying
\[
S_ce_i=-\kappa_i(c)e_i,
\]
write
\[
d_{t,i}
:=
\nabla_{\partial C}d_t(c)\cdot e_i.
\]
Then
\begin{equation}
(\gamma_t)_{ij}
=
\bigl(1+d_t\kappa_i\bigr)^2\delta_{ij}
+
d_{t,i}d_{t,j}.
\label{eq:moving-metric-principal-frame}
\end{equation}

The exact inverse metric can also be written explicitly.

\begin{proposition}[Exact inverse metric]
	\label{prop:exact-inverse-moving-metric}
	Set
	\begin{equation}
	w_t
	:=
	L_t^{-1}\nabla_{\partial C}d_t.
	\label{eq:wt-section5}
	\end{equation}
	Then
	\begin{equation}
	\boxed{
		\gamma_t^{-1}
		=
		L_t^{-2}
		-
		\frac{
			L_t^{-2}\nabla_{\partial C}d_t
			\otimes
			L_t^{-2}\nabla_{\partial C}d_t
		}{
			1+|w_t|^2
		}.
	}
	\label{eq:exact-inverse-metric}
	\end{equation}
\end{proposition}

\begin{proof}
	Equation~\eqref{eq:metric-operator-form} has the form
	\[
	\gamma_t
	=
	L_t^2+p\otimes p,
	\qquad
	p:=\nabla_{\partial C}d_t.
	\]
	Since \(L_t\) is invertible, the rank-one inversion formula gives
	\[
	(L_t^2+p\otimes p)^{-1}
	=
	L_t^{-2}
	-
	\frac{
		L_t^{-2}p\otimes L_t^{-2}p
	}{
		1+p\cdot L_t^{-2}p
	}.
	\]
	Finally,
	\[
	p\cdot L_t^{-2}p
	=
	|L_t^{-1}p|^2
	=
	|w_t|^2.
	\]
\end{proof}

The surface Jacobian follows from the same rank-one structure.

\begin{corollary}[Exact surface Jacobian]
	\label{cor:exact-surface-jacobian}
	Let \(J_t\) denote the Jacobian of
	\(\Psi_t:\partial C\to\partial\Omega_t\). Then
	\begin{equation}
	\boxed{
		J_t(c)
		=
		\det L_t(c)\,
		\sqrt{1+|w_t(c)|^2}.
	}
	\label{eq:surface-jacobian}
	\end{equation}
	Consequently,
	\begin{equation}
	d\mathcal H^{N-1}_{\partial\Omega_t}
	=
	\det L_t\,
	\sqrt{1+|w_t|^2}\,
	d\mathcal H^{N-1}_{\partial C}.
	\label{eq:surface-measure-pullback}
	\end{equation}
\end{corollary}

\begin{proof}
	In a \(\gamma\)-orthonormal frame,
	\[
	\gamma_t=L_t^2+p\otimes p,
	\qquad
	p=\nabla_{\partial C}d_t.
	\]
	The matrix determinant lemma gives
	\[
	\begin{aligned}
	\det\gamma_t
	&=
	\det(L_t^2)
	\left(
	1+p\cdot L_t^{-2}p
	\right)
	\\
	&=
	(\det L_t)^2
	\left(
	1+|w_t|^2
	\right).
	\end{aligned}
	\]
	Because \(L_t>0\),
	\[
	\det L_t>0.
	\]
	Taking square roots gives
	\eqref{eq:surface-jacobian}.
\end{proof}

\begin{remark}[Exact character of the formulas]
	\label{rem:exact-moving-metric}
	All identities in this subsection are exact. No smallness assumption on
	\(d_t\), \(\nabla_{\partial C}d_t\), or \(h_t=d_0-d_t\) has been used.
	
	The small parameter \(t\) enters only through the particular thickness
	profile \(d_t\). The asymptotic regime \(t\downarrow0\) will be introduced
	only after the exact moving-boundary identities have been established.
\end{remark}
\subsection{Exact outward and inward unit normals of the moving boundary}
\label{subsec:exact-moving-normal}

We now derive the unit normal field of the moving hypersurface
\[
\partial\Omega_t=\Gamma_{d_t}
\]
directly from the exact differential obtained in
Proposition~\ref{prop:exact-Psi-t-differential}.

Recall the notation
\[
L_t(c)
:=
I-d_t(c)S_c,
\]
and define
\begin{equation}
w_t(c)
:=
L_t(c)^{-1}
\nabla_{\partial C}d_t(c).
\label{eq:wt-definition-section5}
\end{equation}
We also write
\begin{equation}
\lambda_t(c)
:=
\sqrt{1+|w_t(c)|^2}.
\label{eq:lambda-t-exact-normalization}
\end{equation}

\begin{theorem}[Exact unit normal of the moving boundary]
	\label{thm:exact-moving-normal}
	For every sufficiently small \(t>0\), the outward unit normal to
	\(\partial\Omega_t\) at
	\[
	\Psi_t(c)
	=
	c+d_t(c)\nu(c)
	\]
	is
	\begin{equation}
	\boxed{
		N_t(c)
		=
		\frac{
			\nu(c)-w_t(c)
		}{
			\lambda_t(c)
		}
		=
		\frac{
			\nu(c)
			-
			L_t(c)^{-1}\nabla_{\partial C}d_t(c)
		}{
			\sqrt{
				1+
				\left|
				L_t(c)^{-1}\nabla_{\partial C}d_t(c)
				\right|^2
			}
		}.
	}
	\label{eq:exact-Nt}
	\end{equation}
	Consequently, the inward unit normal is
	\begin{equation}
	\boxed{
		n_t(c)
		=
		-N_t(c)
		=
		\frac{
			-\nu(c)+w_t(c)
		}{
			\lambda_t(c)
		}.
	}
	\label{eq:exact-nt}
	\end{equation}
\end{theorem}

\begin{proof}
	Let
	\[
	\tau\in T_c\partial C.
	\]
	By Proposition~\ref{prop:exact-Psi-t-differential},
	\[
	D\Psi_t(c)[\tau]
	=
	L_t(c)\tau
	+
	\bigl(
	\nabla_{\partial C}d_t(c)\cdot\tau
	\bigr)\nu(c).
	\]
	
	Set
	\[
	W_t(c)
	:=
	\nu(c)-w_t(c).
	\]
	Since
	\[
	w_t(c)
	=
	L_t(c)^{-1}
	\nabla_{\partial C}d_t(c)
	\in T_c\partial C,
	\]
	we obtain
	\begin{align*}
	W_t(c)\cdot D\Psi_t(c)[\tau]
	&=
	\nabla_{\partial C}d_t(c)\cdot\tau
	-
	w_t(c)\cdot L_t(c)\tau.
	\end{align*}
	Because \(L_t(c)\) is symmetric,
	\[
	w_t(c)\cdot L_t(c)\tau
	=
	L_t(c)w_t(c)\cdot\tau
	=
	\nabla_{\partial C}d_t(c)\cdot\tau.
	\]
	Hence
	\[
	W_t(c)\cdot D\Psi_t(c)[\tau]=0.
	\]
	Thus \(W_t(c)\) is normal to \(\partial\Omega_t\).
	
	Moreover,
	\[
	\nu(c)\perp w_t(c),
	\]
	and therefore
	\[
	|W_t(c)|^2
	=
	1+|w_t(c)|^2
	=
	\lambda_t(c)^2.
	\]
	Normalizing gives
	\[
	N_t(c)
	=
	\frac{W_t(c)}{|W_t(c)|}.
	\]
	
	The orientation is outward. Indeed,
	\[
	N_t(c)\cdot\nu(c)
	=
	\frac1{\lambda_t(c)}>0,
	\]
	so \(N_t\) has a positive component in the outward normal direction of
	the core. This proves \eqref{eq:exact-Nt}; equation
	\eqref{eq:exact-nt} follows immediately.
\end{proof}

The preceding computation yields an exact angle identity.

\begin{corollary}[Exact radial transversality]
	\label{cor:exact-radial-transversality}
	For every sufficiently small \(t>0\) and every \(c\in\partial C\),
	\begin{equation}
	\boxed{
		N_t(c)\cdot\nu(c)
		=
		\frac1{\lambda_t(c)}
		=
		\frac{
			1
		}{
			\sqrt{
				1+
				\left|
				L_t(c)^{-1}\nabla_{\partial C}d_t(c)
				\right|^2
			}
		}
		>0.
	}
	\label{eq:exact-radial-transversality}
	\end{equation}
	Equivalently,
	\begin{equation}
	n_t(c)\cdot\nu(c)
	=
	-\frac1{\lambda_t(c)}
	<0.
	\label{eq:exact-inward-radial-component}
	\end{equation}
\end{corollary}

This transversality is a local property of the normal graph. It should not
be confused with the geometric normal property defining
\(\mathcal O_C\), which requires the entire inward normal ray to reach the
interior of the core.

We next relate the exact geometric formula to the elliptic level-set
description.

\begin{proposition}[Compatibility with the level-set gradient]
	\label{prop:normal-gradient-compatibility}
	At the point
	\[
	x_t(c)
	:=
	\Psi_t(c)
	=
	c+d_t(c)\nu(c),
	\]
	the Euclidean gradient of \(u\) satisfies the exact identity
	\begin{equation}
	\boxed{
		\nabla u(x_t(c))
		=
		\partial_rU(c,d_t(c))
		\bigl(
		\nu(c)-w_t(c)
		\bigr).
	}
	\label{eq:gradient-exact-normal-decomposition}
	\end{equation}
	Consequently,
	\begin{equation}
	|\nabla u(x_t(c))|
	=
	-\partial_rU(c,d_t(c))
	\lambda_t(c),
	\label{eq:gradient-magnitude-exact}
	\end{equation}
	and
	\begin{equation}
	\boxed{
		N_t(c)
		=
		-
		\frac{\nabla u(x_t(c))}
		{|\nabla u(x_t(c))|},
		\qquad
		n_t(c)
		=
		\frac{\nabla u(x_t(c))}
		{|\nabla u(x_t(c))|}.
	}
	\label{eq:level-set-normal-orientation}
	\end{equation}
\end{proposition}

\begin{proof}
	Decompose the ambient gradient relative to
	\[
	\mathbb R^N
	=
	T_c\partial C
	\oplus
	\mathbb R\nu(c):
	\]
	\[
	\nabla u(x_t)
	=
	P_c\nabla u(x_t)
	+
	\partial_rU(c,d_t)\nu(c).
	\]
	
	From \eqref{eq:A-geometric-form},
	\[
	A(c,d_t)
	=
	L_t P_c\nabla u(x_t).
	\]
	On the other hand, the exact thickness-gradient identity gives
	\[
	A(c,d_t)
	=
	-\partial_rU(c,d_t)
	\nabla_{\partial C}d_t.
	\]
	Therefore
	\[
	L_tP_c\nabla u(x_t)
	=
	-\partial_rU(c,d_t)
	\nabla_{\partial C}d_t,
	\]
	and hence
	\[
	P_c\nabla u(x_t)
	=
	-\partial_rU(c,d_t)
	L_t^{-1}\nabla_{\partial C}d_t
	=
	-\partial_rU(c,d_t)w_t.
	\]
	Thus
	\[
	\begin{aligned}
	\nabla u(x_t)
	&=
	\partial_rU(c,d_t)\nu
	-
	\partial_rU(c,d_t)w_t
	\\
	&=
	\partial_rU(c,d_t)
	(\nu-w_t),
	\end{aligned}
	\]
	which proves
	\eqref{eq:gradient-exact-normal-decomposition}.
	
	Since
	\[
	\partial_rU(c,d_t)<0
	\]
	in the outer collar,
	\[
	|\nabla u(x_t)|
	=
	-\partial_rU(c,d_t)
	|\nu-w_t|
	=
	-\partial_rU(c,d_t)\lambda_t.
	\]
	Combining this with Theorem~\ref{thm:exact-moving-normal} gives
	\eqref{eq:level-set-normal-orientation}.
\end{proof}

\begin{remark}[No independent PDE remainder in the normal field]
	\label{rem:no-PDE-normal-remainder}
	Once the level boundary has been represented by the thickness function
	\(d_t\), its unit normal is determined completely by
	\[
	d_t,\qquad
	\nabla_{\partial C}d_t,\qquad
	S.
	\]
	The elliptic equation influences the normal field through the profile
	\(d_t\), but it does not generate an additional independent correction
	term in the geometric normal formula.
	
	In particular, the exact identity
	\[
	N_t
	=
	\frac{\nu-L_t^{-1}\nabla_{\partial C}d_t}
	{\sqrt{1+|L_t^{-1}\nabla_{\partial C}d_t|^2}}
	\]
	is the starting point for every subsequent asymptotic expansion.
\end{remark}
\subsection{Convergence of the moving geometry to the background geometry}
\label{subsec:moving-geometry-convergence}

We now compare the exact geometric quantities associated with the moving
boundary
\[
\partial\Omega_t=\Gamma_{d_t}
\]
with those of the fixed outer boundary
\[
\partial\Omega=\Gamma_{d_0}.
\]

Recall the definitions
\[
L_t
=
I-d_tS,
\qquad
L_0
=
I-d_0S,
\]
and
\[
w_t
=
L_t^{-1}\nabla_{\partial C}d_t,
\qquad
w_0
=
L_0^{-1}\nabla_{\partial C}d_0.
\]
The corresponding outward unit normals are
\[
N_t
=
\frac{\nu-w_t}
{\sqrt{1+|w_t|^2}},
\qquad
N_0
=
\frac{\nu-w_0}
{\sqrt{1+|w_0|^2}}.
\]

\begin{proposition}[Convergence of the tangential operators]
	\label{prop:Lt-convergence}
	There exists a constant \(C_L>0\) such that, for all sufficiently small
	\(t>0\),
	\begin{equation}
	\|L_t-L_0\|_{L^\infty(\partial C)}
	\leq
	C_L t.
	\label{eq:Lt-linear-rate}
	\end{equation}
	Moreover,
	\begin{equation}
	\|L_t^{-1}-L_0^{-1}\|_{L^\infty(\partial C)}
	\leq
	C_L t.
	\label{eq:Lt-inverse-linear-rate}
	\end{equation}
\end{proposition}

\begin{proof}
	Since
	\[
	L_t-L_0
	=
	-(d_t-d_0)S,
	\]
	we have
	\[
	\|L_t-L_0\|_{L^\infty}
	\leq
	\|S\|_{L^\infty}
	\|d_t-d_0\|_{L^\infty}.
	\]
	By Theorem~\ref{thm:dt-C1-rate-section4},
	\[
	\|d_t-d_0\|_{L^\infty}
	\leq
	C_dt.
	\]
	Thus \eqref{eq:Lt-linear-rate} follows.
	
	For the inverse, use the resolvent identity
	\[
	L_t^{-1}-L_0^{-1}
	=
	L_t^{-1}
	(L_0-L_t)
	L_0^{-1}.
	\]
	Since the eigenvalues of \(L_t\) and \(L_0\) are bounded below by \(1\),
	\[
	\|L_t^{-1}\|,
	\|L_0^{-1}\|
	\leq1.
	\]
	Therefore
	\[
	\|L_t^{-1}-L_0^{-1}\|
	\leq
	\|L_t-L_0\|,
	\]
	which gives
	\eqref{eq:Lt-inverse-linear-rate}.
\end{proof}

\begin{proposition}[Convergence of the geometric slope field]
	\label{prop:wt-convergence}
	There exists a constant \(C_w>0\) such that, for all sufficiently small
	\(t>0\),
	\begin{equation}
	\|w_t-w_0\|_{L^\infty(\partial C)}
	\leq
	C_wt.
	\label{eq:wt-linear-rate}
	\end{equation}
	Consequently,
	\begin{equation}
	w_t\longrightarrow w_0
	\qquad
	\text{uniformly on }\partial C
	\quad\text{as }t\downarrow0.
	\label{eq:wt-convergence}
	\end{equation}
\end{proposition}

\begin{proof}
	Write
	\[
	\begin{aligned}
	w_t-w_0
	&=
	L_t^{-1}
	\bigl(
	\nabla_{\partial C}d_t
	-
	\nabla_{\partial C}d_0
	\bigr)
	\\
	&\quad+
	\bigl(
	L_t^{-1}-L_0^{-1}
	\bigr)
	\nabla_{\partial C}d_0.
	\end{aligned}
	\]
	Since
	\[
	\|L_t^{-1}\|\leq1,
	\]
	Theorem~\ref{thm:dt-C1-rate-section4} gives
	\[
	\|
	\nabla_{\partial C}d_t
	-
	\nabla_{\partial C}d_0
	\|_{L^\infty}
	\leq
	C_dt.
	\]
	The second term is controlled by
	Proposition~\ref{prop:Lt-convergence}. Hence
	\[
	\|w_t-w_0\|_{L^\infty}
	\leq
	C_wt
	\]
	for a constant \(C_w>0\).
\end{proof}

\begin{proposition}[Convergence of the unit normal fields]
	\label{prop:normal-field-convergence}
	There exists a constant \(C_N>0\) such that
	\begin{equation}
	\boxed{
		\|N_t-N_0\|_{L^\infty(\partial C)}
		\leq
		C_Nt
	}
	\label{eq:Nt-linear-rate-section5}
	\end{equation}
	for all sufficiently small \(t>0\).
	
	Equivalently, for the inward unit normals
	\[
	n_t=-N_t,
	\qquad
	n_0=-N_0,
	\]
	one has
	\begin{equation}
	\|n_t-n_0\|_{L^\infty(\partial C)}
	\leq
	C_Nt.
	\label{eq:nt-linear-rate-section5}
	\end{equation}
\end{proposition}

\begin{proof}
	For each \(c\in\partial C\), define
	\[
	\mathcal G_c(z)
	:=
	\frac{\nu(c)-z}
	{\sqrt{1+|z|^2}}.
	\]
	Then
	\[
	N_t(c)=\mathcal G_c(w_t(c)),
	\qquad
	N_0(c)=\mathcal G_c(w_0(c)).
	\]
	
	The set
	\[
	\{
	w_t(c):
	c\in\partial C,\;
	0<t<t_{\mathrm{col}}
	\}
	\]
	is uniformly bounded for small \(t\), because
	\[
	w_t\longrightarrow w_0
	\]
	uniformly. The map
	\[
	z\longmapsto
	\frac{\nu-z}{\sqrt{1+|z|^2}}
	\]
	is smooth and therefore Lipschitz on bounded subsets of the tangent
	space, with a constant that can be chosen uniformly in \(c\).
	
	Hence
	\[
	|N_t(c)-N_0(c)|
	\leq
	C
	|w_t(c)-w_0(c)|.
	\]
	Using
	\eqref{eq:wt-linear-rate}
	gives
	\eqref{eq:Nt-linear-rate-section5}.
\end{proof}

The exact transversality factor also converges quantitatively.

\begin{corollary}[Convergence of the radial angle]
	\label{cor:radial-angle-convergence}
	There exists a constant \(C_\theta>0\) such that
	\begin{equation}
	\left|
	N_t(c)\cdot\nu(c)
	-
	N_0(c)\cdot\nu(c)
	\right|
	\leq
	C_\theta t
	\label{eq:radial-angle-linear-rate}
	\end{equation}
	uniformly on \(\partial C\).
	
	In particular, if
	\[
	N_0(c)\cdot\nu(c)
	\geq
	\theta_0>0
	\]
	on \(\partial C\), then, after decreasing the small-level threshold if
	necessary,
	\begin{equation}
	N_t(c)\cdot\nu(c)
	\geq
	\frac{\theta_0}{2}
	>0
	\qquad
	\text{for every }c\in\partial C.
	\label{eq:moving-uniform-radial-transversality}
	\end{equation}
\end{corollary}

\begin{proof}
	The estimate follows immediately from
	\eqref{eq:Nt-linear-rate-section5}, because
	\[
	|\nu(c)|=1.
	\]
	The second statement follows from the background lower bound by choosing
	\(t>0\) sufficiently small.
\end{proof}

We also record convergence of the induced metric and surface Jacobian.

\begin{proposition}[Convergence of metric and surface measure]
	\label{prop:metric-jacobian-convergence}
	There exists a constant \(C_G>0\) such that
	\begin{align}
	\|\gamma_t-\gamma_0\|_{L^\infty(\partial C)}
	&\leq
	C_Gt,
	\label{eq:metric-linear-rate}
	\\
	\|J_t-J_0\|_{L^\infty(\partial C)}
	&\leq
	C_Gt.
	\label{eq:jacobian-linear-rate}
	\end{align}
\end{proposition}

\begin{proof}
	From Proposition~\ref{prop:exact-moving-metric},
	\[
	\gamma_t
	=
	L_t^2
	+
	\nabla_{\partial C}d_t
	\otimes
	\nabla_{\partial C}d_t.
	\]
	Since
	\[
	L_t-L_0=O(t)
	\]
	and
	\[
	\nabla_{\partial C}d_t
	-
	\nabla_{\partial C}d_0
	=
	O(t)
	\]
	uniformly, it follows that
	\[
	\gamma_t-\gamma_0=O(t)
	\]
	uniformly.
	
	Similarly,
	\[
	J_t
	=
	\det L_t\sqrt{1+|w_t|^2}.
	\]
	Both factors depend smoothly on \(L_t\) and \(w_t\), which converge to
	\(L_0\) and \(w_0\) at rate \(O(t)\). Hence
	\[
	J_t-J_0=O(t)
	\]
	uniformly.
\end{proof}

\begin{remark}[What this subsection establishes]
	\label{rem:moving-geometry-convergence-scope}
	The conclusions of this subsection are stability statements only:
	\[
	L_t\to L_0,\qquad
	w_t\to w_0,\qquad
	N_t\to N_0,\qquad
	\gamma_t\to\gamma_0,
	\]
	all with linear \(O(t)\) control.
	
	No explicit first-variation coefficient has yet been introduced. Such a
	coefficient will be derived only from the already established
	\(C^1\) expansion
	\[
	d_t=d_0-tb_0+o_{C^1}(t)
	\]
	in the next subsection.
\end{remark}
\subsection{First variation of the moving normal field}
\label{subsec:first-variation-moving-normal}

We now extract the first-order term in the exact normal formula established
in Subsection~\ref{subsec:exact-moving-normal}.

To avoid confusion with the induced metric \(\gamma_t\), we write
\begin{equation}
\lambda_t(c)
:=
\sqrt{1+|w_t(c)|^2},
\qquad
\lambda_0(c)
:=
\sqrt{1+|w_0(c)|^2}.
\label{eq:lambda-t-definition}
\end{equation}
Thus
\begin{equation}
N_t
=
\frac{\nu-w_t}{\lambda_t},
\qquad
N_0
=
\frac{\nu-w_0}{\lambda_0}.
\label{eq:Nt-lambda-form}
\end{equation}

Recall from Proposition~\ref{prop:dt-first-variation-C1} that
\begin{equation}
d_t
=
d_0-tb_0+o_{C^1}(t),
\qquad
b_0=\frac1{a_0}.
\label{eq:dt-first-variation-section5}
\end{equation}

We first determine the induced first variation of the geometric slope field
\[
w_t=L_t^{-1}\nabla_{\partial C}d_t.
\]

\begin{proposition}[First variation of the slope field]
	\label{prop:first-variation-wt}
	Define
	\begin{equation}
	\boxed{
		Z_0
		:=
		L_0^{-1}
		\left(
		\nabla_{\partial C}b_0
		+
		b_0 S w_0
		\right).
	}
	\label{eq:Z0-definition-section5}
	\end{equation}
	Then, as \(t\downarrow0\),
	\begin{equation}
	\boxed{
		w_t
		=
		w_0-tZ_0+o_{L^\infty}(t).
	}
	\label{eq:wt-first-variation-section5}
	\end{equation}
\end{proposition}

\begin{proof}
	From
	\[
	L_t=I-d_tS
	\]
	and \eqref{eq:dt-first-variation-section5}, we obtain
	\begin{equation}
	L_t
	=
	L_0+t\,b_0S+o_{L^\infty}(t).
	\label{eq:Lt-first-variation}
	\end{equation}
	Since \(L_0\) is uniformly invertible,
	\begin{equation}
	L_t^{-1}
	=
	L_0^{-1}
	-
	t\,b_0L_0^{-1}SL_0^{-1}
	+
	o_{L^\infty}(t).
	\label{eq:Lt-inverse-first-variation}
	\end{equation}
	
	Moreover,
	\begin{equation}
	\nabla_{\partial C}d_t
	=
	\nabla_{\partial C}d_0
	-
	t\nabla_{\partial C}b_0
	+
	o_{L^\infty}(t).
	\label{eq:gradient-dt-first-variation-section5}
	\end{equation}
	Therefore
	\begin{align*}
	w_t
	&=
	L_t^{-1}\nabla_{\partial C}d_t
	\\
	&=
	\left(
	L_0^{-1}
	-
	t\,b_0L_0^{-1}SL_0^{-1}
	\right)
	\left(
	\nabla_{\partial C}d_0
	-
	t\nabla_{\partial C}b_0
	\right)
	+
	o(t)
	\\
	&=
	w_0
	-
	tL_0^{-1}\nabla_{\partial C}b_0
	-
	t\,b_0L_0^{-1}S w_0
	+
	o(t),
	\end{align*}
	uniformly on \(\partial C\). This is precisely
	\eqref{eq:wt-first-variation-section5}.
\end{proof}

We can now differentiate the normalization appearing in the exact unit
normal.

\begin{theorem}[First variation of the outward unit normal]
	\label{thm:first-variation-Nt-section5}
	As \(t\downarrow0\),
	\begin{equation}
	\boxed{
		N_t
		=
		N_0+t\,\mathcal N_0+o_{L^\infty}(t),
	}
	\label{eq:Nt-first-variation-section5}
	\end{equation}
	where
	\begin{equation}
	\boxed{
		\mathcal N_0
		=
		\frac1{\lambda_0}
		\bigl(
		I-N_0\otimes N_0
		\bigr)
		Z_0.
	}
	\label{eq:mathcalN0-section5}
	\end{equation}
	
	Equivalently,
	\begin{equation}
	\boxed{
		\mathcal N_0
		=
		\frac1{\lambda_0}
		\bigl(
		I-N_0\otimes N_0
		\bigr)
		L_0^{-1}
		\left(
		\nabla_{\partial C}b_0
		+
		b_0Sw_0
		\right).
	}
	\label{eq:mathcalN0-expanded-section5}
	\end{equation}
\end{theorem}

\begin{proof}
	For fixed \(c\in\partial C\), define
	\begin{equation}
	\mathcal G_c(w)
	:=
	\frac{\nu(c)-w}
	{\sqrt{1+|w|^2}}.
	\label{eq:Gc-normalization}
	\end{equation}
	Then
	\[
	N_t(c)=\mathcal G_c(w_t(c)),
	\qquad
	N_0(c)=\mathcal G_c(w_0(c)).
	\]
	
	The differential of the normalization map at \(w_0\) is
	\begin{equation}
	D\mathcal G_c(w_0)[z]
	=
	-
	\frac1{\lambda_0}
	\bigl(
	I-N_0\otimes N_0
	\bigr)z.
	\label{eq:DGc-section5}
	\end{equation}
	Indeed, the differential of the map
	\[
	v\longmapsto\frac{v}{|v|}
	\]
	at \(v_0=\nu-w_0\) is
	\[
	\frac1{|v_0|}
	\bigl(
	I-N_0\otimes N_0
	\bigr),
	\]
	and the variation of \(v=\nu-w\) is \(-z\).
	
	By Proposition~\ref{prop:first-variation-wt},
	\[
	w_t-w_0
	=
	-tZ_0+o(t).
	\]
	Taylor expansion therefore gives
	\begin{align*}
	N_t-N_0
	&=
	D\mathcal G_c(w_0)
	[w_t-w_0]
	+
	o(t)
	\\
	&=
	\frac{t}{\lambda_0}
	\bigl(
	I-N_0\otimes N_0
	\bigr)Z_0
	+
	o(t),
	\end{align*}
	uniformly on \(\partial C\). This proves
	\eqref{eq:Nt-first-variation-section5}.
\end{proof}

\begin{corollary}[First variation of the inward normal]
	\label{cor:first-variation-inward-normal}
	Since
	\[
	n_t=-N_t,
	\qquad
	n_0=-N_0,
	\]
	one has
	\begin{equation}
	\boxed{
		n_t
		=
		n_0-t\,\mathcal N_0+o_{L^\infty}(t).
	}
	\label{eq:nt-first-variation-section5}
	\end{equation}
\end{corollary}

The variation field has the expected orthogonality property.

\begin{corollary}[Tangency of the normal variation]
	\label{cor:normal-variation-tangent}
	For every \(c\in\partial C\),
	\begin{equation}
	\boxed{
		\mathcal N_0(c)\cdot N_0(c)=0.
	}
	\label{eq:normal-variation-orthogonality}
	\end{equation}
	Thus \(\mathcal N_0(c)\) belongs to the tangent space of the unit sphere
	at \(N_0(c)\).
\end{corollary}

\begin{proof}
	From \eqref{eq:mathcalN0-section5},
	\[
	\mathcal N_0
	=
	\frac1{\lambda_0}
	(I-N_0\otimes N_0)Z_0.
	\]
	The operator
	\[
	I-N_0\otimes N_0
	\]
	is the orthogonal projection onto \(N_0^\perp\). Hence
	\[
	\mathcal N_0\cdot N_0=0.
	\]
\end{proof}

\begin{remark}[Structure of the first variation]
	\label{rem:normal-first-variation-structure}
	The field
	\[
	Z_0
	=
	L_0^{-1}
	\left(
	\nabla_{\partial C}b_0
	+
	b_0Sw_0
	\right)
	\]
	contains two distinct geometric contributions.
	
	The term
	\[
	L_0^{-1}\nabla_{\partial C}b_0
	\]
	comes from spatial variation of the PDE-generated boundary speed
	\(b_0\), whereas
	\[
	b_0L_0^{-1}Sw_0
	\]
	comes from the variation of the geometric operator
	\[
	L_t^{-1}=(I-d_tS)^{-1}
	\]
	under displacement of the boundary.
	
	The projection
	\[
	I-N_0\otimes N_0
	\]
	then removes the component parallel to \(N_0\), as required by the
	unit-length constraint.
\end{remark}

\begin{remark}[Regularity level of the expansion]
	\label{rem:normal-expansion-regularity}
	The expansion
	\[
	N_t
	=
	N_0+t\mathcal N_0+o(t)
	\]
	is asserted in \(L^\infty(\partial C)\). This is the natural consequence
	of the \(C^1\) first variation of \(d_t\), because the exact unit normal
	depends on \(d_t\) and \(\nabla_{\partial C}d_t\).
	
	A corresponding first-order expansion for derivatives of \(N_t\) would
	require control of second tangential derivatives of \(d_t\) and is not
	used at this stage.
\end{remark}
\section{Stability of the class $\mathcal O_C$ under elliptic level deformation}
\label{sec:stability-OC}

\subsection{Persistence of the geometric normal property}
\label{subsec:persistence-GNP}

Let
\[
q_0(c):=F_0(c)
=
\Psi_0(c)+\tau_0(c)n_0(c)
\in\partial C
\]
be the background return point. Let \(\eta_{\mathrm{ret}}>0\) be the
uniform return-transversality constant from
Proposition~\ref{prop:return-map-regularity}. Thus
\begin{equation}
-\,n_0(c)\cdot\nu(q_0(c))
\geq
\eta_{\mathrm{ret}}
\qquad
\text{for every }c\in\partial C.
\label{eq:return-transversality-sec6}
\end{equation}

Let \(\rho_C\) denote the signed distance to \(\partial C\), negative in
\(C\). Choose \(r_C>0\) such that \(\rho_C\) is of class \(C^{2,\alpha}\)
on the \(r_C\)-tubular neighborhood of \(\partial C\). For \(s\geq0\)
sufficiently small,
\[
G(c,s)
:=
\rho_C\bigl(q_0(c)+s n_0(c)\bigr)
\]
is therefore well defined. Since
\[
G(c,0)=0
\]
and
\[
\partial_sG(c,0)
=
\nu(q_0(c))\cdot n_0(c)
\leq
-\eta_{\mathrm{ret}},
\]
compactness gives a number
\[
0<\delta_*<r_C
\]
such that
\begin{equation}
\rho_C\bigl(q_0(c)+s n_0(c)\bigr)
\leq
-\frac{\eta_{\mathrm{ret}}}{2}s
\label{eq:background-penetration}
\end{equation}
for
\[
c\in\partial C,
\qquad
0\leq s\leq\delta_*.
\]

Set
\begin{equation}
y_0(c)
:=
q_0(c)+\delta_*n_0(c)
=
\Psi_0(c)
+
\bigl(\tau_0(c)+\delta_*\bigr)n_0(c).
\label{eq:y0-sec6}
\end{equation}
Then
\begin{equation}
\rho_C(y_0(c))
\leq
-m_0,
\qquad
m_0:=
\frac{\eta_{\mathrm{ret}}\delta_*}{2}>0.
\label{eq:uniform-interior-margin}
\end{equation}

\begin{theorem}[Stability of the geometric normal property]
	\label{thm:PDE-levels-in-OC}
	There exists \(t_{\mathcal O}>0\) such that
	\begin{equation}
	\boxed{
		\Omega_t\in\mathcal O_C
		\qquad
		\text{for every }0<t<t_{\mathcal O}.
	}
	\label{eq:Omega-t-in-OC}
	\end{equation}
\end{theorem}

\begin{proof}
	For \(t>0\) sufficiently small, define
	\begin{equation}
	y_t(c)
	:=
	\Psi_t(c)
	+
	\bigl(\tau_0(c)+\delta_*\bigr)n_t(c).
	\label{eq:yt-sec6}
	\end{equation}
	This point lies on the inward normal ray issued from
	\(\Psi_t(c)\in\partial\Omega_t\).
	
	By Theorem~\ref{thm:dt-C1-rate-section4},
	\[
	\|\Psi_t-\Psi_0\|_{L^\infty}
	=
	\|d_t-d_0\|_{L^\infty}
	=
	O(t),
	\]
	and Proposition~\ref{prop:normal-field-convergence} gives
	\[
	\|n_t-n_0\|_{L^\infty}
	=
	O(t).
	\]
	Since \(\tau_0\) is bounded on the compact manifold \(\partial C\),
	\begin{align}
	|y_t(c)-y_0(c)|
	&\leq
	|\Psi_t(c)-\Psi_0(c)|
	\notag\\
	&\quad+
	\bigl(\tau_0(c)+\delta_*\bigr)
	|n_t(c)-n_0(c)|
	\notag\\
	&\leq Ct
	\label{eq:yt-y0-estimate}
	\end{align}
	uniformly in \(c\).
	
	By \eqref{eq:uniform-interior-margin},
	\[
	\operatorname{dist}
	\bigl(y_0(c),\partial C\bigr)
	\geq m_0.
	\]
	Choose \(t_{\mathcal O}>0\) smaller than \(t_{\mathrm{col}}\) and
small enough that all preceding estimates for \(\Psi_t\) and \(n_t\)
hold on \((0,t_{\mathcal O})\), and in addition
\[
Ct_{\mathcal O}<\frac{m_0}{2}.
\]
Then, for \(0<t<t_{\mathcal O}\),
	\[
	y_t(c)\in\operatorname{int}C
	\qquad
	\text{for every }c\in\partial C.
	\]
	Hence every inward normal ray from \(\partial\Omega_t\) meets
	\(\operatorname{int}C\). This proves condition \emph{(OC2)}.
	
	Condition \emph{(OC1)} follows from the global normal-graph representation
	proved in Theorem~\ref{thm:global-superlevel-normal-graph} and
	Proposition~\ref{prop:superlevel-domain-generated-by-dt}:
	\[
	\Omega_t
	=
	C\cup
	\left\{
	\Phi(c,r):
	c\in\partial C,\;
	0<r<d_t(c)
	\right\}.
	\]
	Finally, Proposition~\ref{prop:small-level-localization} gives
	\[
	\partial\Omega_t\in C^{2,\alpha}.
	\]
	Therefore
	\[
	\Omega_t\in\mathcal O_C.
	\]
\end{proof}
\section{Stability and asymptotics of the return map}
\label{sec:return-stability}

\subsection{Quantitative stability of the return map}
\label{subsec:return-map-stability}

By Theorem~\ref{thm:PDE-levels-in-OC}, there exists
\(t_{\mathcal O}>0\) such that
\[
\Omega_t\in\mathcal O_C,
\qquad
0<t<t_{\mathcal O}.
\]
Hence the return construction is well defined for every such \(t\).

For \(c\in\partial C\), let
\[
\Psi_t(c)
=
c+d_t(c)\nu(c)
\in\partial\Omega_t
\]
and let \(n_t(c)\) denote the inward unit normal to
\(\partial\Omega_t\) at \(\Psi_t(c)\). Define
\begin{equation}
\tau_t(c)
:=
\min\left\{
s\geq0:
\Psi_t(c)+s\,n_t(c)\in C
\right\}.
\label{eq:tau-t-definition}
\end{equation}
The corresponding return map is
\begin{equation}
F_t(c)
:=
\Psi_t(c)+\tau_t(c)n_t(c)
\in\partial C.
\label{eq:Ft-definition}
\end{equation}

For the background domain,
\[
F_0(c)
=
\Psi_0(c)+\tau_0(c)n_0(c).
\]

Let \(\rho_C\) be the signed distance to \(\partial C\), negative in
\(C\), and define
\begin{equation}
H_t(c,s)
:=
\rho_C\bigl(\Psi_t(c)+s\,n_t(c)\bigr).
\label{eq:Ht-return}
\end{equation}
Then
\[
H_t(c,\tau_t(c))=0.
\]

For the background return,
\[
H_0(c,\tau_0(c))=0,
\]
and the uniform return transversality gives
\begin{equation}
\partial_sH_0(c,\tau_0(c))
=
\nu(F_0(c))\cdot n_0(c)
\leq
-\eta_{\mathrm{ret}}
<0.
\label{eq:H0-return-transversality}
\end{equation}

\begin{theorem}[Linear stability of the return time and return map]
	\label{thm:return-map-linear-stability}
	There exist \(0<t_R\leq t_{\mathcal O}\) and \(C_R>0\) such that, for every
	\(0<t<t_R\),
	\begin{equation}
	\boxed{
		\|\tau_t-\tau_0\|_{L^\infty(\partial C)}
		\leq
		C_R t
	}
	\label{eq:tau-linear-stability}
	\end{equation}
	and
	\begin{equation}
	\boxed{
		\|F_t-F_0\|_{L^\infty(\partial C)}
		\leq
		C_R t.
	}
	\label{eq:F-linear-stability}
	\end{equation}
	Consequently,
	\begin{equation}
	F_t\longrightarrow F_0
	\qquad
	\text{uniformly on }\partial C
	\quad\text{as }t\downarrow0.
	\label{eq:Ft-to-F0}
	\end{equation}
\end{theorem}

\begin{proof}
	Since \(\tau_0\) is continuous and strictly positive on the compact
	manifold \(\partial C\), set
	\[
	\tau_-:=\min_{c\in\partial C}\tau_0(c)>0.
	\]
	By continuity of \(\partial_sH_0\), compactness of \(\partial C\), and
	the uniform transversality estimate
	\eqref{eq:H0-return-transversality}, there exists
	\[
	0<\delta_R<\frac{\tau_-}{2}
	\]
	such that
	\begin{equation}
	\partial_sH_0(c,s)
	\leq
	-\frac{3}{4}\eta_{\mathrm{ret}}
	\label{eq:H0-return-collar}
	\end{equation}
	whenever
	\[
	|s-\tau_0(c)|\leq2\delta_R.
	\]
	Shrinking \(\delta_R\) further if necessary, all corresponding points
	lie in a fixed tubular neighborhood of \(\partial C\) on which
	\(\rho_C\) is \(C^2\).
	
	From Theorem~\ref{thm:dt-C1-rate-section4} and
	Proposition~\ref{prop:normal-field-convergence},
	\begin{equation}
	\|\Psi_t-\Psi_0\|_{L^\infty}
	+
	\|n_t-n_0\|_{L^\infty}
	\leq Ct.
	\label{eq:Psi-n-stability-return}
	\end{equation}
	Hence, uniformly for
	\[
	|s-\tau_0(c)|\leq2\delta_R,
	\]
	\begin{equation}
	\left|
	\bigl(\Psi_t(c)+s n_t(c)\bigr)
	-
	\bigl(\Psi_0(c)+s n_0(c)\bigr)
	\right|
	\leq Ct.
	\label{eq:return-ray-comparison}
	\end{equation}
	Since \(\rho_C\) is Lipschitz,
	\begin{equation}
	|H_t(c,s)-H_0(c,s)|
	\leq Ct.
	\label{eq:Ht-H0-C0}
	\end{equation}
	
	Moreover,
	\[
	\partial_sH_t(c,s)
	=
	\nabla\rho_C
	\bigl(\Psi_t(c)+s n_t(c)\bigr)
	\cdot n_t(c).
	\]
	Since \(\nabla\rho_C\) is Lipschitz in the chosen tubular neighborhood,
	\eqref{eq:return-ray-comparison} implies
	\begin{equation}
	|\partial_sH_t(c,s)-\partial_sH_0(c,s)|
	\leq Ct.
	\label{eq:Ht-H0-derivative}
	\end{equation}
	Therefore, after decreasing \(t_R\) if necessary,
	\begin{equation}
	\partial_sH_t(c,s)
	\leq
	-\frac12\eta_{\mathrm{ret}}
	\label{eq:Ht-uniform-monotonicity}
	\end{equation}
	throughout the return collar.
	
	From \eqref{eq:H0-return-collar},
	\[
	H_0(c,\tau_0(c)-\delta_R)
	\geq
	\frac34\eta_{\mathrm{ret}}\delta_R,
	\]
	whereas
	\[
	H_0(c,\tau_0(c)+\delta_R)
	\leq
	-\frac34\eta_{\mathrm{ret}}\delta_R.
	\]
	Using \eqref{eq:Ht-H0-C0}, for sufficiently small \(t\),
	\[
	H_t(c,\tau_0(c)-\delta_R)>0,
	\qquad
	H_t(c,\tau_0(c)+\delta_R)<0.
	\]
	Together with \eqref{eq:Ht-uniform-monotonicity}, this gives a unique zero
of \(H_t(c,\cdot)\) in the interval
\[
(\tau_0(c)-\delta_R,\tau_0(c)+\delta_R).
\]
At this zero one has \(\partial_sH_t<0\), so the ray crosses from the
exterior of \(C\) into its interior. Since the intersection of a ray
with the convex body \(C\) is an interval, such an exterior-to-interior
crossing is necessarily its left endpoint. Hence this zero is precisely
the first entrance time \(\tau_t(c)\).

At \(s=\tau_0(c)\),
	\[
	|H_t(c,\tau_0(c))|
	\leq Ct.
	\]
	Since
	\[
	H_t(c,\tau_t(c))=0,
	\]
	the mean value theorem and
	\eqref{eq:Ht-uniform-monotonicity} yield
	\[
	|\tau_t(c)-\tau_0(c)|
	\leq
	\frac{2C}{\eta_{\mathrm{ret}}}\,t.
	\]
	This proves \eqref{eq:tau-linear-stability}.
	
	Finally,
	\begin{align*}
	F_t-F_0
	&=
	\Psi_t-\Psi_0
	+
	(\tau_t-\tau_0)n_t
	+
	\tau_0(n_t-n_0).
	\end{align*}
	The first and third terms are \(O(t)\) by
	\eqref{eq:Psi-n-stability-return}, and the second is \(O(t)\) by
	\eqref{eq:tau-linear-stability}. Hence
	\[
	\|F_t-F_0\|_{L^\infty}
	\leq C_R t.
	\]
\end{proof}

\subsection{First variation of the return time}
\label{subsec:first-variation-return-time}

Recall
\[
F_t(c)
=
\Psi_t(c)+\tau_t(c)n_t(c),
\qquad
F_0(c)
=
\Psi_0(c)+\tau_0(c)n_0(c).
\]
From Proposition~\ref{prop:dt-first-variation-C1},
\begin{equation}
\Psi_t
=
\Psi_0-tb_0\nu+o_{L^\infty}(t),
\label{eq:Psi-first-variation-sec7}
\end{equation}
and Corollary~\ref{cor:first-variation-inward-normal} gives
\begin{equation}
n_t
=
n_0-t\mathcal N_0+o_{L^\infty}(t).
\label{eq:n-first-variation-sec7}
\end{equation}

Let
\[
q_0(c):=F_0(c)\in\partial C.
\]

\begin{theorem}[First variation of the return time]
	\label{thm:first-variation-return-time}
	As \(t\downarrow0\),
	\begin{equation}
	\boxed{
		\tau_t
		=
		\tau_0+t\vartheta_0+o_{L^\infty}(t),
	}
	\label{eq:tau-first-variation}
	\end{equation}
	where
	\begin{equation}
	\boxed{
		\vartheta_0(c)
		=
		\frac{
			b_0(c)\,
			\nu(q_0(c))\cdot\nu(c)
			+
			\tau_0(c)\,
			\nu(q_0(c))\cdot\mathcal N_0(c)
		}{
			\nu(q_0(c))\cdot n_0(c)
		}.
	}
	\label{eq:vartheta0}
	\end{equation}
	The denominator is uniformly nonzero:
	\[
	\nu(q_0(c))\cdot n_0(c)
	\leq
	-\eta_{\mathrm{ret}}<0.
	\]
\end{theorem}

\begin{proof}
	By Theorem~\ref{thm:return-map-linear-stability},
	\[
	\tau_t-\tau_0=O(t)
	\]
	uniformly on \(\partial C\).
	
	Using \eqref{eq:Psi-first-variation-sec7} and
	\eqref{eq:n-first-variation-sec7},
	\begin{align}
	F_t-F_0
	&=
	(\Psi_t-\Psi_0)
	+
	(\tau_t-\tau_0)n_0
	+
	\tau_0(n_t-n_0)
	\notag\\
	&\qquad+
	(\tau_t-\tau_0)(n_t-n_0)
	\notag\\
	&=
	-tb_0\nu
	+
	(\tau_t-\tau_0)n_0
	-
	t\tau_0\mathcal N_0
	+
	o(t)
	\label{eq:Ft-prevariation}
	\end{align}
	uniformly on \(\partial C\).
	
	Since
	\[
	F_t(c),F_0(c)\in\partial C,
	\]
	the signed distance \(\rho_C\) satisfies
	\[
	\rho_C(F_t(c))
	=
	\rho_C(F_0(c))
	=
	0.
	\]
	Taylor expansion at \(q_0(c)=F_0(c)\) yields
	\begin{equation}
	0
	=
	\nu(q_0(c))
	\cdot
	\bigl(F_t(c)-F_0(c)\bigr)
	+
	O\bigl(|F_t(c)-F_0(c)|^2\bigr).
	\label{eq:rhoC-return-expansion}
	\end{equation}
	By \eqref{eq:F-linear-stability},
	\[
	|F_t-F_0|=O(t),
	\]
	so the remainder in \eqref{eq:rhoC-return-expansion} is \(O(t^2)\).
	
	Substitution of \eqref{eq:Ft-prevariation} gives
	\begin{align*}
	0
	&=
	-t b_0\,
	\nu(q_0)\cdot\nu
	\\
	&\quad+
	(\tau_t-\tau_0)\,
	\nu(q_0)\cdot n_0
	\\
	&\quad-
	t\tau_0\,
	\nu(q_0)\cdot\mathcal N_0
	+
	o(t).
	\end{align*}
	Therefore
	\begin{align}
	\frac{\tau_t-\tau_0}{t}
	&=
	\frac{
		b_0\,\nu(q_0)\cdot\nu
		+
		\tau_0\,\nu(q_0)\cdot\mathcal N_0
	}{
		\nu(q_0)\cdot n_0
	}
	+
	o(1).
	\label{eq:tau-difference-quotient}
	\end{align}
	The return transversality bound permits uniform division by the
	denominator, proving \eqref{eq:tau-first-variation}.
\end{proof}
\subsection{First variation of the return map}
\label{subsec:first-variation-return-map}

We combine the expansions
\[
\Psi_t
=
\Psi_0-tb_0\nu+o_{L^\infty}(t),
\]
\[
n_t
=
n_0-t\mathcal N_0+o_{L^\infty}(t),
\]
and
\[
\tau_t
=
\tau_0+t\vartheta_0+o_{L^\infty}(t).
\]

\begin{theorem}[First variation of the return map]
	\label{thm:first-variation-return-map}
	As \(t\downarrow0\),
	\begin{equation}
	\boxed{
		F_t
		=
		F_0+t\mathcal A_0+o_{L^\infty}(t),
	}
	\label{eq:Ft-first-variation}
	\end{equation}
	where
	\begin{equation}
	\boxed{
		\mathcal A_0
		=
		-b_0\nu
		+
		\vartheta_0 n_0
		-
		\tau_0\mathcal N_0.
	}
	\label{eq:A0-basic}
	\end{equation}
	Here
	\begin{equation}
	\vartheta_0
	=
	\frac{
		b_0\,\nu(F_0)\cdot\nu
		+
		\tau_0\,\nu(F_0)\cdot\mathcal N_0
	}{
		\nu(F_0)\cdot n_0
	}.
	\label{eq:vartheta0-repeat}
	\end{equation}
\end{theorem}

\begin{proof}
	Using
	\[
	F_t=\Psi_t+\tau_t n_t,
	\]
	we obtain
	\begin{align*}
	F_t
	&=
	\Psi_0-tb_0\nu
	+
	(\tau_0+t\vartheta_0)
	(n_0-t\mathcal N_0)
	+
	o(t)
	\\
	&=
	\Psi_0+\tau_0n_0
	\\
	&\quad+
	t\left(
	-b_0\nu
	+
	\vartheta_0n_0
	-
	\tau_0\mathcal N_0
	\right)
	+
	o(t).
	\end{align*}
	Since
	\[
	F_0=\Psi_0+\tau_0n_0,
	\]
	this proves \eqref{eq:Ft-first-variation}.
\end{proof}

The coefficient \(\mathcal A_0(c)\) is necessarily tangent to the core
boundary at the return point.

\begin{corollary}[Tangency of the return-map variation]
	\label{cor:A0-tangent}
	For every \(c\in\partial C\),
	\begin{equation}
	\boxed{
		\mathcal A_0(c)
		\in
		T_{F_0(c)}\partial C.
	}
	\label{eq:A0-tangent}
	\end{equation}
	Equivalently,
	\begin{equation}
	\nu(F_0(c))\cdot\mathcal A_0(c)=0.
	\label{eq:A0-normal-zero}
	\end{equation}
\end{corollary}

\begin{proof}
	Taking the scalar product of \eqref{eq:A0-basic} with
	\(\nu(F_0)\) gives
	\begin{align*}
	\nu(F_0)\cdot\mathcal A_0
	&=
	-b_0\,\nu(F_0)\cdot\nu
	\\
	&\quad+
	\vartheta_0\,\nu(F_0)\cdot n_0
	\\
	&\quad-
	\tau_0\,\nu(F_0)\cdot\mathcal N_0.
	\end{align*}
	Substituting \eqref{eq:vartheta0-repeat} gives zero.
\end{proof}

There is a useful geometric form of the variation. Define
\begin{equation}
V_0
:=
-b_0\nu-\tau_0\mathcal N_0.
\label{eq:V0-definition}
\end{equation}
Then
\[
\vartheta_0
=
-
\frac{
	\nu(F_0)\cdot V_0
}{
	\nu(F_0)\cdot n_0
}.
\]
Hence
\begin{equation}
\boxed{
	\mathcal A_0
	=
	V_0
	-
	\frac{
		\nu(F_0)\cdot V_0
	}{
		\nu(F_0)\cdot n_0
	}
	n_0.
}
\label{eq:A0-oblique-projection}
\end{equation}

Thus \(\mathcal A_0\) is obtained from the unconstrained first-order
displacement
\[
V_0=-b_0\nu-\tau_0\mathcal N_0
\]
by projection onto \(T_{F_0(c)}\partial C\) along the return direction
\(n_0(c)\).

Equivalently, if
\begin{equation}
\Pi_0^{\mathrm{ret}}(c)v
:=
v
-
\frac{
	\nu(F_0(c))\cdot v
}{
	\nu(F_0(c))\cdot n_0(c)
}
n_0(c),
\label{eq:return-oblique-projection}
\end{equation}
then
\begin{equation}
\boxed{
	\mathcal A_0
	=
	\Pi_0^{\mathrm{ret}}
	\left(
	-b_0\nu-\tau_0\mathcal N_0
	\right).
}
\label{eq:A0-projection-form}
\end{equation}
\section{A PDE-generated non-autonomous return system}
\label{sec:nonautonomous-return-system}

\subsection{Asymptotic autonomy and summable perturbation defects}
\label{subsec:asymptotic-autonomy}

Let
\[
t_n\downarrow0
\]
be a sequence of positive levels. After discarding finitely many initial
terms, we may assume
\[
0<t_n<t_R
\]
for every \(n\), where \(t_R\) is the threshold in
Theorem~\ref{thm:return-map-linear-stability}.

We consider the non-autonomous discrete system
\begin{equation}
x_{n+1}
=
F_{t_n}(x_n),
\qquad
x_n\in\partial C.
\label{eq:nonautonomous-system}
\end{equation}

Let \(d_C\) denote the intrinsic distance on the compact hypersurface
\(\partial C\). Since the intrinsic and ambient distances are uniformly
equivalent at sufficiently small scales, the estimate
\[
\|F_t-F_0\|_{L^\infty}=O(t)
\]
obtained in Section~\ref{sec:return-stability} implies
\begin{equation}
\sup_{c\in\partial C}
d_C\bigl(F_t(c),F_0(c)\bigr)
\leq
C_{\mathrm{na}}t
\label{eq:Ft-F0-intrinsic}
\end{equation}
for all sufficiently small \(t>0\).

\begin{theorem}[Asymptotic autonomy]
	\label{thm:asymptotic-autonomy}
	The sequence of maps \(F_{t_n}\) converges uniformly to \(F_0\):
	\begin{equation}
	\boxed{
		\sup_{c\in\partial C}
		d_C\bigl(F_{t_n}(c),F_0(c)\bigr)
		\leq
		C_{\mathrm{na}}t_n
		\longrightarrow0.
	}
	\label{eq:nonautonomous-map-convergence}
	\end{equation}
	Consequently, every orbit of \eqref{eq:nonautonomous-system} satisfies
	\begin{equation}
	\boxed{
		d_C\bigl(x_{n+1},F_0(x_n)\bigr)
		\leq
		C_{\mathrm{na}}t_n.
	}
	\label{eq:pseudo-orbit-defect}
	\end{equation}
	Thus the non-autonomous orbit is an asymptotic pseudo-orbit of the
	background return map \(F_0\).
\end{theorem}

\begin{proof}
	Equation~\eqref{eq:nonautonomous-map-convergence} follows directly from
	\eqref{eq:Ft-F0-intrinsic} with \(t=t_n\). Since
	\[
	x_{n+1}=F_{t_n}(x_n),
	\]
	we obtain
	\[
	\begin{aligned}
	d_C\bigl(x_{n+1},F_0(x_n)\bigr)
	&=
	d_C\bigl(F_{t_n}(x_n),F_0(x_n)\bigr)
	\\
	&\leq
	\sup_{c\in\partial C}
	d_C\bigl(F_{t_n}(c),F_0(c)\bigr)
	\\
	&\leq
	C_{\mathrm{na}}t_n.
	\end{aligned}
	\]
\end{proof}

\begin{corollary}[Summable perturbation regime]
	\label{cor:summable-perturbation}
	If
	\begin{equation}
	\sum_{n=0}^{\infty} t_n<\infty,
	\label{eq:summable-levels}
	\end{equation}
	then
	\begin{equation}
	\boxed{
		\sum_{n=0}^{\infty}
		\sup_{c\in\partial C}
		d_C\bigl(F_{t_n}(c),F_0(c)\bigr)
		<\infty.
	}
	\label{eq:summable-map-defects}
	\end{equation}
	In particular, every orbit of the non-autonomous system satisfies
	\begin{equation}
	\boxed{
		\sum_{n=0}^{\infty}
		d_C\bigl(x_{n+1},F_0(x_n)\bigr)
		<\infty.
	}
	\label{eq:summable-orbit-defects}
	\end{equation}
\end{corollary}

\begin{proof}
	By \eqref{eq:Ft-F0-intrinsic},
	\[
	\sup_{c\in\partial C}
	d_C\bigl(F_{t_n}(c),F_0(c)\bigr)
	\leq
	C_{\mathrm{na}}t_n.
	\]
	Summation gives \eqref{eq:summable-map-defects}, and
	\eqref{eq:summable-orbit-defects} follows from
	\eqref{eq:pseudo-orbit-defect}.
\end{proof}

\begin{corollary}[Dyadic level sequence]
\label{cor:dyadic-levels}
Fix
\[
0<\bar t<t_R
\]
and set
\begin{equation}
t_n=\bar t\,2^{-n}.
\label{eq:dyadic-level-sequence}
\end{equation}
Then
\begin{equation}
\sup_{c\in\partial C}
d_C\bigl(F_{\bar t2^{-n}}(c),F_0(c)\bigr)
=
O(2^{-n}),
\label{eq:dyadic-map-rate}
\end{equation}
and
\begin{equation}
\sum_{n=0}^{\infty}
\sup_{c\in\partial C}
d_C\bigl(F_{\bar t2^{-n}}(c),F_0(c)\bigr)
<\infty.
\label{eq:dyadic-summability}
\end{equation}
\end{corollary}
\subsection{Finite-horizon orbit tracking}
\label{subsec:finite-horizon-tracking}

Since
\[
F_0\in C^{1,\alpha}(\partial C,\partial C)
\]
and \(\partial C\) is compact, there exists \(L_0\geq0\) such that
\begin{equation}
d_C\bigl(F_0(x),F_0(y)\bigr)
\leq
L_0\,d_C(x,y)
\qquad
\text{for all }x,y\in\partial C.
\label{eq:F0-Lipschitz}
\end{equation}

Fix \(N\geq0\). Let
\begin{equation}
x_{n+1}=F_{t_n}(x_n),
\qquad n\geq N,
\label{eq:nonautonomous-orbit-N}
\end{equation}
and let
\begin{equation}
y_{n+1}=F_0(y_n),
\qquad
y_N=x_N.
\label{eq:autonomous-comparison-orbit}
\end{equation}

Thus the autonomous orbit is restarted from the same point as the
non-autonomous orbit at time \(N\).

\begin{theorem}[Finite-horizon tracking]
	\label{thm:finite-horizon-tracking}
	There exists \(C_{\mathrm{na}}>0\), independent of \(N\), such that for
	every \(m\geq1\),
	\begin{equation}
	\boxed{
		d_C(x_{N+m},y_{N+m})
		\leq
		C_{\mathrm{na}}
		\sum_{j=0}^{m-1}
		L_0^{\,m-1-j}\,
		t_{N+j}.
	}
	\label{eq:finite-horizon-tracking}
	\end{equation}
	
	Consequently, for every fixed \(m\geq1\),
	\begin{equation}
	\boxed{
		d_C(x_{N+m},y_{N+m})
		\longrightarrow0
		\qquad
		\text{as }N\to\infty.
	}
	\label{eq:finite-horizon-limit}
	\end{equation}
\end{theorem}

\begin{proof}
	Set
	\[
	e_n:=d_C(x_n,y_n).
	\]
	Since \(x_N=y_N\),
	\[
	e_N=0.
	\]
	
	Using the triangle inequality,
	\begin{align*}
	e_{n+1}
	&=
	d_C\bigl(F_{t_n}(x_n),F_0(y_n)\bigr)
	\\
	&\leq
	d_C\bigl(F_{t_n}(x_n),F_0(x_n)\bigr)
	+
	d_C\bigl(F_0(x_n),F_0(y_n)\bigr).
	\end{align*}
	By \eqref{eq:nonautonomous-map-convergence} and
	\eqref{eq:F0-Lipschitz},
	\begin{equation}
	e_{n+1}
	\leq
	C_{\mathrm{na}}t_n
	+
	L_0e_n.
	\label{eq:tracking-recurrence}
	\end{equation}
	
	Iterating \eqref{eq:tracking-recurrence} from \(e_N=0\) gives
	\[
	e_{N+m}
	\leq
	C_{\mathrm{na}}
	\sum_{j=0}^{m-1}
	L_0^{\,m-1-j}t_{N+j},
	\]
	which proves \eqref{eq:finite-horizon-tracking}.
	
	For fixed \(m\), the sum contains only finitely many terms, and
	\[
	t_{N+j}\longrightarrow0
	\qquad
	\text{as }N\to\infty
	\]
	for every \(j=0,\ldots,m-1\). Hence
	\[
	e_{N+m}\longrightarrow0.
	\]
\end{proof}

\begin{corollary}[Dyadic tracking estimate]
\label{cor:dyadic-finite-horizon-tracking}
If
\[
t_n=\bar t\,2^{-n},
\qquad
0<\bar t<t_R,
\]
then for every \(m\geq1\),
\begin{equation}
d_C(x_{N+m},y_{N+m})
\leq
C_{\mathrm{na}}\,\bar t\,2^{-N}
\sum_{j=0}^{m-1}
L_0^{\,m-1-j}2^{-j}.
\label{eq:dyadic-finite-horizon}
\end{equation}
In particular, for each fixed \(m\),
\begin{equation}
d_C(x_{N+m},y_{N+m})
=
O(2^{-N})
\qquad
\text{as }N\to\infty.
\label{eq:dyadic-finite-horizon-rate}
\end{equation}
\end{corollary}
\subsection{Long-time tracking in a contraction regime}
\label{subsec:contraction-tracking}

The finite-horizon estimate of Subsection~\ref{subsec:finite-horizon-tracking}
does not by itself imply long-time shadowing. A uniform conclusion
requires additional dynamical stability of the background map. We record
one natural sufficient condition.

\begin{theorem}[Long-time tracking on an invariant contraction region]
\label{thm:contraction-tracking}
Let \(K\subset\partial C\) be a nonempty compact set. Assume that
\begin{enumerate}[label=\textnormal{(C\arabic*)}]
\item \(F_0(K)\subset K\);
\item there exists \(q\in(0,1)\) such that
\begin{equation}
d_C(F_0(x),F_0(y))
\leq
q\,d_C(x,y)
\qquad
\text{for all }x,y\in K;
\label{eq:K-contraction}
\end{equation}
\item there exists \(t_K>0\) such that \(F_t(K)\subset K\) for every
\(0<t<t_K\).
\end{enumerate}
Let \((t_n)\) satisfy
\[
0<t_n<\min\{t_R,t_K\},
\qquad
t_n\to0,
\]
and let
\[
x_{n+1}=F_{t_n}(x_n),
\qquad
y_{n+1}=F_0(y_n),
\qquad
x_0,y_0\in K.
\]
Then, with \(e_n:=d_C(x_n,y_n)\),
\begin{equation}
e_n
\leq
q^n e_0
+
C_{\mathrm{na}}
\sum_{k=0}^{n-1}q^{\,n-1-k}t_k.
\label{eq:contraction-tracking-bound}
\end{equation}
In particular,
\begin{equation}
e_n\longrightarrow0.
\label{eq:contraction-tracking-limit}
\end{equation}
Moreover, \(F_0\) has a unique fixed point \(p\in K\), and both
\(y_n\) and \(x_n\) converge to \(p\).
\end{theorem}

\begin{proof}
By the invariance assumptions, both orbits remain in \(K\). Combining
\eqref{eq:nonautonomous-map-convergence} with
\eqref{eq:K-contraction} gives
\[
e_{n+1}
\leq
C_{\mathrm{na}}t_n+q e_n.
\]
Iteration yields \eqref{eq:contraction-tracking-bound}.

To prove convergence of the convolution term, fix \(\varepsilon>0\) and
choose \(N\) so that
\[
t_k\leq\frac{\varepsilon(1-q)}{2C_{\mathrm{na}}}
\qquad\text{for }k\geq N.
\]
The contribution of the finitely many indices \(k<N\) tends to zero as
\(n\to\infty\), whereas the remaining tail is bounded by
\[
C_{\mathrm{na}}
\sum_{k=N}^{n-1}q^{n-1-k}t_k
\leq
\frac{\varepsilon}{2}.
\]
Together with \(q^ne_0\to0\), this proves
\eqref{eq:contraction-tracking-limit}.

Since \(K\) is compact, it is complete for \(d_C\). Assumptions
\textnormal{(C1)}--\textnormal{(C2)} therefore allow the Banach fixed-point
theorem to be applied to \(F_0|_K\), yielding a unique fixed point
\(p\in K\) and \(y_n\to p\). Finally,
\[
d_C(x_n,p)
\leq
e_n+d_C(y_n,p)\longrightarrow0.
\]
\end{proof}

\begin{corollary}[Dyadic forcing in the contraction regime]
\label{cor:dyadic-contraction-tracking}
Under the assumptions of Theorem~\ref{thm:contraction-tracking}, let
\[
t_n=\bar t\,2^{-n},
\qquad
0<\bar t<\min\{t_R,t_K\}.
\]
If \(x_0=y_0\), then
\begin{equation}
d_C(x_n,y_n)
\leq
C_{\mathrm{na}}\bar t
\sum_{k=0}^{n-1}q^{\,n-1-k}2^{-k}.
\label{eq:dyadic-contraction-convolution}
\end{equation}
Consequently,
\[
d_C(x_n,y_n)
=
\begin{cases}
O(q^n), & q>\frac12,\\[1mm]
O(n2^{-n}), & q=\frac12,\\[1mm]
O(2^{-n}), & 0<q<\frac12.
\end{cases}
\]
\end{corollary}

\begin{remark}[Scope of the long-time conclusion]
\label{rem:no-general-long-time-shadowing}
Without a contraction, shadowing, hyperbolicity, or another comparable
stability mechanism, the estimates established in this paper give
asymptotic pseudo-orbit control and finite-horizon tracking, but not a
general long-time shadowing theorem.
\end{remark}

\section{Explicit benchmark and computational protocol}
\label{sec:numerical-illustrations}

\subsection{A radial benchmark}
\label{subsec:radial-benchmark}

We first consider a benchmark in dimension \(N=2\) for which the outer
boundary, the level boundaries, and the return map can be written
explicitly.

Let
\[
C=B(0,1),
\qquad
\Omega=B(0,R),
\qquad
R>1.
\]
Let \(f\in C_c^\infty(B(0,r_s))\) be radial, nonnegative, and not
identically zero, with
\[
0<r_s<1.
\]
Denote its total mass by
\[
M:=\int_{\Omega} f(x)\,dx>0.
\]

By radial symmetry, the solution of
\[
\begin{cases}
-\Delta u=f & \text{in }\Omega,\\
u=0 & \text{on }\partial\Omega
\end{cases}
\]
is radial. Since \(\operatorname{supp}f\subset B(0,r_s)\subset C\), the
function \(u\) is harmonic in the annulus
\[
\{x\in\mathbb R^2:1\le |x|\le R\}.
\]
Hence, in this annulus,
\begin{equation}
u(r)=\frac{M}{2\pi}\log\frac{R}{r},
\qquad
1\le r\le R.
\label{eq:radial-u-shell}
\end{equation}

For every sufficiently small \(t>0\), the level set \(\{u=t\}\) is the
circle of radius
\begin{equation}
r_t
=
R\,e^{-2\pi t/M}.
\label{eq:radial-level-radius}
\end{equation}
Therefore
\begin{equation}
d_0=R-1,
\qquad
d_t=r_t-1,
\qquad
h_t=d_0-d_t=R-r_t.
\label{eq:radial-thicknesses}
\end{equation}
Using \eqref{eq:radial-level-radius},
\begin{equation}
h_t
=
R\left(1-e^{-2\pi t/M}\right)
=
\frac{2\pi R}{M}\,t+O(t^2).
\label{eq:radial-h-expansion}
\end{equation}

Moreover,
\[
\partial_r u(R)
=
-\frac{M}{2\pi R},
\]
so that
\begin{equation}
a_0
=
-\partial_rU(d_0)
=
\frac{M}{2\pi R},
\qquad
b_0=\frac{1}{a_0}=\frac{2\pi R}{M}.
\label{eq:radial-a0-b0}
\end{equation}
Thus \eqref{eq:radial-h-expansion} is exactly the prediction
\[
h_t=t\,b_0+O(t^2).
\]

In this radial geometry the outward unit normal is purely radial, and the
inward normal ray from \(\partial\Omega_t\) returns to the same point of
\(\partial C\). Hence
\begin{equation}
F_t=\operatorname{Id}_{\partial C}
\qquad
\text{for every sufficiently small }t\ge0.
\label{eq:radial-return-map}
\end{equation}
In particular,
\[
F_0=\operatorname{Id}_{\partial C},
\qquad
\tau_t=d_t,
\]
and the return-map perturbation vanishes identically:
\[
F_t-F_0\equiv0.
\]

This benchmark verifies explicitly the three basic conclusions of the
theory:
\[
d_t\to d_0,
\qquad
\frac{h_t}{t}\to b_0,
\qquad
F_t\to F_0.
\]
\subsection{A non-radial normal-graph protocol}
\label{subsec:nonradial-example}

We next consider a genuinely non-radial geometry in dimension \(N=2\).

Let
\[
C=B(0,1),
\]
and define the outer boundary by the normal graph
\begin{equation}
\partial\Omega=\Gamma_{d_0},
\qquad
d_0(\theta)=d_\ast+\varepsilon\cos(2\theta),
\label{eq:nonradial-d0}
\end{equation}
with parameters, for example,
\[
d_\ast=1,
\qquad
\varepsilon=0.20.
\]
Equivalently, in polar form,
\begin{equation}
\rho_0(\theta)=1+d_0(\theta)=2+0.20\cos(2\theta).
\label{eq:nonradial-rho0}
\end{equation}
We take \(\Omega:=\Omega[d_0]\). Since
\[
d_0(\theta)\geq0.8>0,
\]
condition \textnormal{(OC1)} holds by construction.

This curve is smooth and strictly convex. Indeed, for a polar curve
\(r=r(\theta)\), strict convexity follows from
\[
r^2+2(r')^2-r r''>0.
\]
Writing \(c=\cos(2\theta)\), the present choice gives
\[
r^2+2(r')^2-r r''
=
4.32+2.4c-0.12c^2
\geq1.8>0.
\]
The geometric normal property can also be checked directly. The outward
unit normal to the polar curve is
\[
N(\theta)
=
\frac{r e_r-r'e_\theta}{\sqrt{r^2+(r')^2}},
\]
so the distance from the origin to the normal line through
\(r(\theta)e_r(\theta)\) is
\[
\frac{r|r'|}{\sqrt{r^2+(r')^2}}
\leq |r'|
\leq0.4<1.
\]
Hence every such normal line intersects \(\operatorname{int}B(0,1)\).
Moreover, for the inward direction \(n=-N\),
\[
r e_r\cdot n
=
-\frac{r^2}{\sqrt{r^2+(r')^2}}<0,
\]
so the closest point to the origin occurs in the forward inward ray.
Thus condition \textnormal{(OC2)} also holds, and therefore
\(\Omega\in\mathcal O_C\).

As source term we choose a smooth radial bump supported well inside the
core, for instance
\begin{equation}
f(x)
=
\begin{cases}
\displaystyle
\exp\!\left(-\frac{1}{1-16|x|^2}\right), & |x|<1/4,\\[2mm]
0, & |x|\ge1/4.
\end{cases}
\label{eq:nonradial-source}
\end{equation}
Then \(\operatorname{supp}f\Subset \operatorname{int}C\).

A reproducible computational protocol is the following.

\begin{enumerate}
	\item Solve
	\[
	\begin{cases}
	-\Delta u=f & \text{in }\Omega,\\
	u=0 & \text{on }\partial\Omega
	\end{cases}
	\]
	by finite elements on a triangulation fitted to \(\Omega\).
	
	\item Choose \(0<\bar t<t_R\) and use a sequence of small levels
	\[
	t_k=\bar t\,2^{-k},\qquad k=0,\dots,6.
	\]
	determine \(d_{t_k}(\theta)\) from the scalar equation
	\[
	u\bigl((1+d_{t_k}(\theta))e_r(\theta)\bigr)=t_k
	\]
	by one-dimensional root finding along each core normal ray.
	
	\item Compute the inward normal field \(n_{t_k}\) on
	\(\Gamma_{d_{t_k}}\), and determine the corresponding return map
	\(F_{t_k}\) by intersecting each inward normal ray with \(\partial C\).
	
	\item Compare \(d_{t_k}\) with \(d_0\), and \(F_{t_k}\) with the
	background return map \(F_0\).
\end{enumerate}

For an implementation of this protocol, natural diagnostics are
\begin{equation}
E_d(t)
:=
\|d_t-d_0\|_{L^\infty(\partial C)},
\qquad
E_F(t)
:=
\|F_t-F_0\|_{L^\infty(\partial C)},
\label{eq:numerical-errors}
\end{equation}
together with the normalized quantities
\begin{equation}
H_t(\theta)
:=
\frac{d_0(\theta)-d_t(\theta)}{t},
\qquad
A_t(\theta)
:=
\frac{F_t(\theta)-F_0(\theta)}{t}.
\label{eq:numerical-normalized}
\end{equation}

The theory predicts
\begin{equation}
E_d(t)=O(t),
\qquad
E_F(t)=O(t),
\label{eq:numerical-prediction-errors}
\end{equation}
and
\begin{equation}
H_t\to b_0
\quad\text{in }C^1(\partial C),
\qquad
A_t\to\mathcal A_0
\quad\text{in }L^\infty(\partial C).
\label{eq:numerical-prediction-normalized}
\end{equation}

In contrast with the radial benchmark, the return map \(F_0\) is now
typically nontrivial:
\[
F_0\neq \operatorname{Id}_{\partial C}.
\]
This non-radial configuration therefore provides a concrete test geometry for the
first-order predictions of the theory; no numerical convergence data are asserted here.
\begin{figure}[H]
\centering
\begin{tikzpicture}[
    scale=1.02,
    >=Latex,
    line cap=round,
    line join=round,
    every node/.style={font=\small}
]
\def\rin{1.45}
\def\ang{34}
\draw[thick] (0,0) circle (\rin);
\fill[gray!20] (0,0) circle (0.34);
\node[below=3pt] at (0,-0.34) {\scriptsize $\operatorname{supp}f$};
\node at (-0.55,0.65) {$C$};
\draw[thick]
plot[smooth cycle, samples=180, domain=0:360, variable=\t]
({(3.05+0.38*cos(2*\t))*cos(\t)},
 {(3.05+0.38*cos(2*\t))*sin(\t)});
\draw[dashed,thick]
plot[smooth cycle, samples=180, domain=0:360, variable=\t]
({(2.45+0.24*cos(2*\t))*cos(\t)},
 {(2.45+0.24*cos(2*\t))*sin(\t)});
\node[anchor=south] at (0,3.62) {$\partial\Omega=\Gamma_{d_0}$};
\node[anchor=east] at (-2.35,2.20) {$\partial\Omega_t=\Gamma_{d_t}$};
\coordinate (c) at ({\rin*cos(\ang)},{\rin*sin(\ang)});
\pgfmathsetmacro{\Rzero}{3.05+0.38*cos(2*\ang)}
\coordinate (psi0) at ({\Rzero*cos(\ang)},{\Rzero*sin(\ang)});
\pgfmathsetmacro{\Rt}{2.45+0.24*cos(2*\ang)}
\coordinate (psit) at ({\Rt*cos(\ang)},{\Rt*sin(\ang)});
\coordinate (Fzero) at ({\rin*cos(5)},{\rin*sin(5)});
\coordinate (Ft) at ({\rin*cos(22)},{\rin*sin(22)});
\draw[blue!70!black,thick,->] (c) -- ($(c)!1.07!(psi0)$);
\fill (c) circle (1.5pt);
\fill (psit) circle (1.5pt);
\fill (psi0) circle (1.5pt);
\node[anchor=south east,fill=white,inner sep=1pt] at ($(c)+(-0.18,0.14)$) {$c$};
\node[anchor=south east,fill=white,inner sep=1pt] at ($(psit)+(-0.18,0.18)$) {$\Psi_t(c)$};
\node[anchor=south west,fill=white,inner sep=1pt] at ($(psi0)+(0.14,0.12)$) {$\Psi_0(c)$};
\node[blue!70!black,anchor=east,fill=white,inner sep=1pt] (nulab) at (1.05,2.12) {$\nu(c)$};
\draw[blue!70!black,thin] (nulab.east) -- ($(c)!0.46!(psit)$);
\draw[red!75!black,thick,->] (psi0) -- (Fzero);
\draw[red!75!black,thick,dashed,->] (psit) -- (Ft);
\fill (Fzero) circle (1.5pt);
\fill (Ft) circle (1.5pt);
\node[anchor=north west,fill=white,inner sep=1pt] at ($(Fzero)+(0.18,-0.16)$) {$F_0(c)$};
\node[anchor=north east,fill=white,inner sep=1pt] at ($(Ft)+(-0.24,-0.06)$) {$F_t(c)$};
\node[red!75!black,anchor=east,fill=white,inner sep=1pt] (ntlab) at (0.95,1.30) {$n_t(c)$};
\draw[red!75!black,thin] (ntlab.east) -- ($(psit)!0.55!(Ft)$);
\node[red!75!black,anchor=west,fill=white,inner sep=1pt] (n0lab) at (2.42,0.82) {$n_0(c)$};
\draw[red!75!black,thin] (n0lab.west) -- ($(psi0)!0.56!(Fzero)$);
\end{tikzpicture}
\caption{Schematic geometry of the non-radial example. For a fixed
$c\in\partial C$, the points $\Psi_t(c)$ and $\Psi_0(c)$ lie on the same
outward normal ray of the core. The corresponding inward normals intersect
$\partial C$ at $F_t(c)$ and $F_0(c)$, respectively.}
\label{fig:nonradial-geometry}
\end{figure}
\section{Open problems and extensions}
\label{sec:open-problems}

The preceding analysis is local in the level parameter \(t\downarrow0\).
Several natural questions remain beyond this boundary-layer regime.

\subsection{Continuation away from the small-level regime}
\label{subsec:global-continuation}

The present argument yields
\[
\Omega_t\in\mathcal O_C
\]
only for sufficiently small \(t>0\). A first question is to determine the
maximal interval
\[
0<t<t_{\max}
\]
on which the superlevel domains retain the geometric normal property.

Possible loss of membership in \(\mathcal O_C\) may occur through several
distinct mechanisms:
\[
\partial_r u(c,d_t(c))=0,
\]
failure of the global normal-graph representation, or loss of
transversality of the inward normal return to \(\partial C\).

A natural problem is therefore to characterize
\[
t_{\max}
:=
\sup\left\{
t>0:
\Omega_s\in\mathcal O_C
\text{ for every }0<s<t
\right\}
\]
directly in terms of the elliptic solution and the geometry of the pair
\((C,\Omega)\).

\subsection{Higher-order asymptotics}
\label{subsec:higher-order-asymptotics}

The first-order expansions
\[
d_t=d_0-tb_0+o(t)
\]
and
\[
F_t=F_0+t\mathcal A_0+o(t)
\]
suggest the existence of higher-order expansions under stronger
regularity assumptions.

One may ask whether
\begin{equation}
d_t
=
d_0-tb_0+t^2b_1+o(t^2)
\label{eq:second-order-dt-open}
\end{equation}
and
\begin{equation}
F_t
=
F_0+t\mathcal A_0+t^2\mathcal A_1+o(t^2)
\label{eq:second-order-Ft-open}
\end{equation}
hold in appropriate topologies.

The coefficients \(b_1\) and \(\mathcal A_1\) should involve the second
normal derivative of \(u\), the curvature of \(\partial C\), and the
differential of the background return construction.

\subsection{Dynamical consequences of the first variation}
\label{subsec:dynamical-first-variation}

The expansion
\[
F_t
=
F_0+t\mathcal A_0+o(t)
\]
describes the infinitesimal deformation of the return dynamics.

An open problem is to determine which dynamical quantities associated
with \(F_0\) admit corresponding first-order perturbation formulas. In
particular, one may study the variation of periodic points, multipliers,
invariant sets, or other structures of the return dynamics under the
PDE-generated deformation.

For a periodic point \(c_0\) of period \(m\),
\[
F_0^m(c_0)=c_0,
\]
the basic question is whether there exist nearby periodic points
\(c_t\) satisfying
\[
F_t^m(c_t)=c_t
\]
and, when they exist, how their displacement is determined by
\(\mathcal A_0\).

\subsection{Beyond strictly convex cores}
\label{subsec:beyond-strict-convexity}

Strict convexity of \(C\) provides global exterior normal coordinates and
a unique metric projection. It is natural to ask which parts of the
theory survive when \(C\) is replaced by a set of positive reach \cite{Federer1959}.

In that setting the normal-coordinate map is available only up to the
reach, and the operator
\[
I-rS
\]
need not remain invertible for arbitrary \(r>0\). The corresponding
problem is to formulate a local version of the class \(\mathcal O_C\)
for which the superlevel deformation and return construction remain
well posed.

A further question concerns nonconvex reference sets, where uniqueness
of the normal representation and uniqueness of the return point can
both fail.

\subsection{Nonlinear elliptic equations}
\label{subsec:nonlinear-elliptic}

The argument near \(\partial\Omega\) uses primarily boundary regularity,
the Hopf lemma, and a nonvanishing derivative along the core-normal
direction. This suggests extensions to nonlinear elliptic equations.

For example, one may consider
\[
-\operatorname{div}
\bigl(A(x,u,\nabla u)\bigr)
=
f
\]
or quasilinear equations such as the \(p\)-Laplace equation
\[
-\Delta_p u=f.
\]

The essential questions are whether one retains a uniform boundary
collar satisfying
\[
\partial_r u<0
\]
and whether the resulting level surfaces possess sufficient regularity
to recover the normal-graph and return-map analysis.

\subsection{Inverse reconstruction from a family of level returns}
\label{subsec:inverse-level-returns}

The family
\[
\{F_t\}_{0<t<t_{\mathcal O}}
\]
contains information about both the elliptic solution and the geometry
of the outer boundary.

At first order,
\[
\frac{F_t-F_0}{t}
\longrightarrow
\mathcal A_0,
\]
while
\[
\mathcal A_0
=
\Pi_0^{\mathrm{ret}}
\left(
-b_0\nu-\tau_0\mathcal N_0
\right).
\]
Since
\[
b_0(c)
=
\frac{1}{
	-\partial_{\nu_\Omega}u(\Psi_0(c))
	\,
	\nu_\Omega(\Psi_0(c))\cdot\nu(c)
},
\]
one may ask which parts of the boundary flux, the thickness profile
\(d_0\), or the source \(f\) can be reconstructed from measurements of
the family \(F_t\).

This leads to a PDE-generated inverse problem for return maps, distinct
from the purely geometric inverse problem for a single fixed return map studied in \cite{BarkatouElMorsalaniInv2026}.

\section{Conclusion}
\label{sec:conclusion}

We have established a boundary-layer mechanism linking an elliptic Dirichlet problem to the return-map geometry of the class \(\mathcal O_C\).

For the superlevel domains
\[
\Omega_t=\{u>t\},
\qquad
t\downarrow0,
\]
the relevant geometric variable is
\[
h_t=d_0-d_t,
\]
and the analysis shows that
\[
d_t\longrightarrow d_0
\qquad\text{in }C^1(\partial C),
\]
with the first-order expansion
\[
d_t=d_0-tb_0+o_{C^1}(t).
\]
The coefficient \(b_0\) is determined by the boundary flux of the elliptic solution and the transversality between the outer boundary and the normal rays of the core.

Using the exact normal-graph geometry, we proved that the moving unit normals converge to the background normal field and obtained their first variation. This control is sufficient to preserve the geometric normal property: for all sufficiently small \(t>0\),
\[
\Omega_t\in\mathcal O_C.
\]
The return map \(F_t\) is therefore well defined by the same geometric construction as the background map \(F_0\).

The return dynamics satisfies the quantitative stability estimate
\[
\|F_t-F_0\|_{L^\infty(\partial C)}
=O(t),
\]
and, more precisely,
\[
F_t
=
F_0+t\mathcal A_0+o_{L^\infty}(t),
\]
where the coefficient \(\mathcal A_0\) is obtained from the first-order displacement of the moving boundary and the first variation of its inward normal, followed by the return transversality correction at \(\partial C\).

Finally, for sequences \(t_n\downarrow0\), the maps \(F_{t_n}\) form an asymptotically autonomous perturbation of \(F_0\). In particular,
\[
d_C\bigl(F_{t_n}(c),F_0(c)\bigr)
=O(t_n)
\]
uniformly in \(c\), and summable level sequences generate summable one-step perturbation defects. Finite-horizon orbit tracking follows without additional dynamical assumptions, while long-time tracking depends on the stability properties of the background map.

The resulting picture is therefore
\[
\boxed{
	-\Delta u=f,\quad u|_{\partial\Omega}=0
	\Longrightarrow
	d_t\to d_0
	\Longrightarrow
	\Omega_t\in\mathcal O_C
	\Longrightarrow
	F_t\to F_0
}
\]
with quantitative first-order information at each stage. This provides a direct bridge between elliptic boundary-layer geometry and the return dynamics associated with domains in \(\mathcal O_C\) surrounding a convex core.

\end{document}